\documentclass[11pt]{amsart}
\usepackage{amsfonts}
\usepackage{amscd}
\usepackage{amssymb}
\usepackage{amsmath}
\usepackage{enumerate}
\usepackage{epsfig}
\RequirePackage{color}

\usepackage{tikz}
\usepackage{tikz-cd}

\usetikzlibrary{matrix,arrows}

\newtheorem{theorem}{Theorem}[section]
\newtheorem{lemma}[theorem]{Lemma}
\newtheorem{proposition}[theorem]{Proposition}

\newtheorem{remark}[theorem]{Remark}

\newcommand{\mfr}[1]{\mathfrak{#1}}
\newcommand{\mr}[1]{\mathrm{#1}}

\newcommand{\ZZ}{\mathbb{Z}}
\newcommand{\FF}{\mathbb{F}}

\newcommand{\QQ}{\mathbb{Q}}
\newcommand{\A}{\mathbb{A}}
\newcommand{\PP}{\mathbb{P}}
\newcommand{\C}{\mathbb{C}}

\newcommand{\bv}{{\bf{v}}}

\newcommand{\Proj}{\mathbb{P}}

\newcommand{\alt}{\mathrm{alt}}

\newcommand{\Ker}{\mathrm{Ker}}

\newcommand{\Sym}{\mathrm{Sym}}

\newcommand{\Char}{\mathop{\mathrm{Char}}}
\newcommand{\disc}{\mathrm{disc}}

\begin{document}
\title{Genus 3  curves whose Jacobians have endomorphisms by  $\QQ (\zeta _7 +\bar{\zeta}_7 )$, II }

\author{J. W. Hoffman}

\address{Department of Mathematics \\
              Louisiana State University \\
              Baton Rouge, Louisiana 70803}

\author{Dun  Liang}
\address{Department of Mathematics \\
              Louisiana State University \\
              Baton Rouge, Louisiana 70803}

 \author{Zhibin Liang}
\address{School of Mathematical Sciences \\
Capital Normal  University \\
Beijing, China, 100048\\
and Beijing International Center for Mathematical Research,\\
 Peking University
}

\author{Ryotaro Okazaki}
\address{Department of Knowledge Engineering and Computer Sciences, Doshisha University, Kyoto-Fu, 610-03 JAPAN}

 \author{Yukiko Sakai}
\address{Department of Mathematics,
College of Liberal Arts and Sciences,
Kitasato University,
Kanagawa 252-0373
JAPAN}

 \author{Haohao Wang}
\address{Department of Mathematics \\
Southeast Missouri State University \\
Cape Girardeau, MO 63701}

\email{hoffman@math.lsu.edu, dliang1@lsu.edu, liangzhb@gmail.com}
 \email{rokazaki@mail.doshisha.ac.jp, y-sakai@kitasato-u.ac.jp,  hwang@semo.edu}

\subjclass[2000]{Primary: 11G10, 11G15, 14H45} \keywords{curves of
genus three, real multiplication, abelian variety}

\begin{abstract}
In this work we consider constructions of genus three curves $X$
such that $\mr{End}(\mr{Jac} (X))\otimes \QQ$ contains the totally real
cubic number field $\QQ (\zeta _7 +\bar{\zeta}_7 )$. We construct explicit three-dimensional
families whose generic member is a nonhyperelliptic genus 3 curve with this
property. The case  when $X$ is hyperelliptic was
studied in \cite{HW}, and some nonhyperelliptic curves were constructed in
\cite{HLSW}.
\end{abstract}

\maketitle

\section{Introduction}
\label{S:intro}

Let $\mfr{M}_g$ be the (coarse) moduli space of projective smooth
curves of genus $g$, and let $\mfr{A}_g$ be the moduli space of principally polarized abelian varieties of dimension $g$.  The Torelli map $X \mapsto \mr{Jac}(X): \mfr{M}_g \to \mfr{A}_g$ is an injection. The study of subvarieties of
$\mfr{A}_g$ defined by the condition that the corresponding abelian varieties $A$ shall have nontrivial endomorphism
rings is an old subject which in its modern form is a part of the theory of Shimura varieties
(\cite{gS}, \cite{gS2}). For this paper we consider varieties
over fields of characteristic 0, so nontrivial means that  $\mr{End}(A)\otimes \QQ$ is larger than $\QQ$. In general,
$\mr{End}(A)\otimes \QQ$ is a semisimple algebra of finite dimension with involution  (see \cite[\S X]{aW} and
\cite{dM}).

When $g=2$ this study was initiated by Humbert in the late 19th century, who investigated the algebraic surfaces
 $\mathcal{H}_R$ parametrizing $(X, \theta)$, where $X$ is a smooth projective curve of genus 2, and
 $\theta$ is an embedding of an order  $R $ in a real quadratic field $\QQ (\sqrt{d})$ into
 $\mr{End}(\mr {Jac}(X))$.  Note that $\mfr{M}_2$ and $\mfr{A}_2$ are birationally equivalent, and that every genus two
 curve is hyperelliptic, so representable by an equation $y^2 = f(x)$ with a polynomial $f$ of degree 5 or 6. Humbert
 accomplished two things: he gave conditions on the analytic moduli $\tau \in \mathfrak{H}_2$ in the Siegel space
 of degree 2 for an abelian variety to have endomorphisms by $R$, and he managed in several cases to give explicit conditions
 on the coefficients of $f$ for the Jacobian of this curve to have endomorphisms by $R$, in effect describing   $\mathcal{H}_R$
 concretely along with its universal family of curves. On the second point: these constructions are related to the classical Poncelet theorems about pairs of conics, and were reinterpreted and generalized in the language of elliptic curves by Mestre (see \cite{jfM}, \cite{jfM2}, and also Sakai's thesis \cite{yS}). These
 Humbert surfaces are special instances of Hilbert modular surfaces (see \cite{vdG}). They play an important role in the geometry
 of the Siegel modular threefolds $\mfr{A}_2 (N) \to \mfr{A}_2$ that are coverings defined by level $N$ structures (see, e.g.,
 \cite{HH}). The problem of explicitly describing abelian varieties of dimension 2 with special endomorphism rings has been considered by many people. For instance see the work of Runge (\cite{bR}), implemented as Magma algorithms by Gruenewald
 (\cite{dG}). Another construction, based on an idea of Dolgachev, which utilizes the theory of $K3$ surfaces, has been made
 into an algorithm
 in the thesis of A. Kumar, and many explicit examples are now known: see the paper of Elkies and Kumar (\cite{EK}).
 For the curves of genus 2 whose endomorphsm algebra contains an order $R$ in a quaternion algebra over $\QQ$, the
 corresponding variety
  $\mathcal{H}_R$ is now a Shimura curve. For explicit computations of these, see Elkies (\cite{nE1}, \cite{nE2}, \cite{nE3}), Hashimoto and Murabayashii (\cite{HM}),
  and Voight (\cite{jV1}, \cite{jV2}, who with D. Kohel implemented much of this in Magma).

 By contrast to the situation in genus 2, the case of genus 3 and higher is much less understood. While it is easy
 to see that the  general member of $\mfr{A}_g$ has trivial endomorphism algebra, i.e.,  $\mr{End}(A)\otimes \QQ=\QQ$,
 it is true but not easy to show that the Jacobian of a  general member of  $\mfr{M}_g$ has trivial endomorphism algebra.
 This is a theorem of Severi and Zariski. When $g=3$, both  $\mfr{A}_3$ and  $\mfr{M}_3$ are birationally equivalent, but now not
 every curve is hyperelliptic:  $\mfr{M}_3 ^{\mr{hyper}}\subset \mfr{M}_3$ is a 5-dimensional
irreducible subvariety of the 6-dimensional  $\mfr{M}_3$. In this paper we give explicit constructions of families of
nonhyperelliptic curves of genus 3 whose Jacobians have endomorphism algebras containing the maximal order
$R = \ZZ[\zeta _7 ^+]$  in the totally
real cubic number field $\QQ (\zeta _7 ^+) := \QQ (\zeta _7 +\bar{\zeta}_7 )$, where $\zeta _7$ is a primitive 7th root of unity.
In earlier papers (\cite{HW}, \cite{HLSW}) we studied the hyperelliptic case, and gave a special construction generalizing
Mestre's method. In this paper we give families with maximal modular dimension. That is, the corresponding moduli space
$\mathcal{H}_R$ is now a Hilbert modular variety of dimension 3, so our family has three independent moduli.

Our method is based on Ellenberg's paper (\cite{JE}), which apparently is based on an earlier paper of Shimada (\cite{Shi}).  Many of the calculations in this paper were carried out with Mathematica (\cite{Math}), Magma (\cite{Magma}), PARI/GP (\cite{pari}) and Sage (\cite{Sage}).

Outline: In section 2 we recall well-known facts about branched coverings of Riemann surfaces. In section 3 we give the main construction. 
Section 4 is devoted to justifying the plane model of our curves stated in Proposition \ref{P:3}. This depends on analyzing the action of the 
dihedral group $D_7$ on the cohomology  of the curve $X$ of genus 8 that covers  our genus 3 curve. In section 5 we show that the curves in our family 
have three independent moduli. Section 6 constructs a $D_7$-invariant principal polarization on the Hodge structure associated to the genus 8 curve. 
Section 7 gives a precise analysis of the conditions under which our constructions give genus 3 curves of the type we are considering. Finally 
explicit equations for the genus 3 curves are in section 8.

\noindent {\bf Acknowledgements:}
{We would like to thank Dr. J. F. Mestre for his helpful comments during his visit to Beijing in summer 2012.  This work was
conducted during an invited academic visit to Beijing International Center for Mathematical Research and the Chinese Academy
of Science, and we would like to thank the hosting institutions, in particular Zhibin Liang and Lihong Zhi, for the invitations, and the hospitality during the visit.  The first author is supported in part by NSA grant 115-60-5012 and NSF grant OISE-1318015; the second author is supported by NSFC11001183 and NSFC1171231, and the last author would like to thank the GRFC grant from Southeast Missouri State University.}

\section{Branched coverings}
\label{S:br}
This section recalls the facts and defines our conventions.
If $X$ is a ``good'' connected topological space, and $\ast\in X$ is a point, we let
$\pi = \pi _1 (X,\ast)$ be the fundamental group. We follow the convention that composition
$\alpha \beta$ of elements $\alpha, \beta \in \pi _1 (X,\ast) $ represented by loops
$a, b: [0, 1] \to X$ is the loop $ab$ with $a$ following $b$.
We let $\tilde{X}$ be the universal covering. This is
a left $\pi$-torsor (i.e., $\pi$ acts on the left making it into a locally trivial  $\pi$-bundle).  Our convention is that
groups $G$ act on the left of spaces $X$. This has the inconvenience that $G$ acts on the right on spaces
of functions $\varphi: X \to \C$ via $(\varphi g)(x) := \varphi (gx)$. Recall the equivalence of categories:
\[
\left \{ \text{right } \pi \text{-sets} \  S\right \} \leftrightarrow \left \{ \text{unramified coverings\ } f: Y \to X\right \} .
\]
The arrow $\leftarrow$ sends $f: Y \to X $ to $S = f^{-1}(\ast)$ with the right monodromy action of
$\pi _1 (X,\ast)$. The arrow $\rightarrow$ sends $S$ to $S\times \tilde{X}/\sim$ where the equivalence relation
is $(s \gamma , x) \sim (s, \gamma x)$ for all $\gamma \in \pi$. Under this equivalence, transitive $\pi$-sets $S$
correspond to connected coverings $Y$. For a transitive set,  choosing any point $s \in S$, the map
$\pi_s \gamma \mapsto s \gamma :\pi _s\backslash \pi  \to S$ is a bijection of right $\pi$-sets, where
$\pi _s = \mr{Stab}_{\pi}(s)$ is the stabilizer of $s$. Also, automorphisms of coverings $ f: Y \to X$ correspond to
automorphisms of the $\pi$-set $ f^{-1}(\ast)$ (i.e., bijections commuting with the given $\pi$-action
on the latter).

Recall that a connected covering $ f: Y \to X$ is Galois with group $G$ if equivalently
\begin{itemize}
\item[1.] $\pi _s$ is a normal subgroup of $\pi$ for all $s\in S$.
\item[2.] $\mr{Aut} (Y/X) = G$ and the degree of $Y \to X$ is $\# G$.
\end{itemize}
Necessarily $G \cong \pi _s\backslash \pi  = \pi/\pi _s$. Under the bijection
$S = \pi _s\backslash \pi$, the automorphisms of $G $ correspond to the left multiplications
of $S = \pi _s\backslash \pi = \pi/\pi _s$ by elements of $\pi$ and these commute with the right
multiplications that define the $\pi$-set $\pi _s\backslash \pi$. The action of $G$ on $Y$ is in the left.
Here is an alternative: we start with a surjective homomorphism: $\varphi : \pi \to G$. Then this
defines a Galois $G$ covering whose right $\pi$-set is $S = G = \pi /\rm{Ker}(\varphi)$, and whose
$\pi$-set automorphisms are the left multiplications by elements of $\pi$. The stabilizers $\pi _s = \rm{Ker}(\varphi)$.

\begin{lemma}
\label{L:1}
Given two surjective homomorphisms  $\varphi _1, \varphi _2 : \pi \to G$, the corresponding right
$\pi$-sets are  isomorphic if and only if there is a group-automorphism $\psi$ of $G$ such that
$\varphi _2 = \psi \varphi _1 $. The set of isomorphisms are the maps
$x \to g \psi (x)$ for some element $g \in G$.
\end{lemma}

\begin{proof}
The last statement follows from the first because the set of automorphism of each of these $\pi$-sets is
given by left multiplications by elements of $G$. An isomorphism of $G$-sets is a bijection
$\psi : G \to G$ with the property that $\psi(g)\varphi _2(x) = \psi  (g \varphi _1(x))$ for all $g \in G, x \in \pi$.
Because the $\varphi $ are surjective, we can write $g = \varphi _1 (\gamma)$ for some
$\gamma \in \pi$.
If $\psi$ is a homomorphism, and $\varphi _2 = \psi \varphi _1 $, then
 $$\psi(g)\varphi _2(x) = \psi(\varphi _1 (\gamma))\varphi _2(x)
 =    \psi(\varphi _1 (\gamma))    \psi(\varphi _1 (x) )  = \psi (\varphi _1 (\gamma)\varphi _1 (x) )
 =  \psi  (g \varphi _1(x)).$$
 On the other hand, left multiplication by $G$ is a transitive action on these sets, so without loss of
 generality, we may assume that $e = \varphi (e)$ : $e$ the identity element of $G$. Then the equation
  $\psi(g)\varphi _2(x) = \psi  (g \varphi _1(x))$ with $g=e$ gives
  $\varphi _2 (\gamma)= \psi \varphi _1 (\gamma)$ for all $\gamma \in \pi$, and writing
  $u = \varphi _1(\gamma), \ v = \varphi _1 (\delta)$, we get
  \[
  \psi(u v) = \psi ( \varphi _1(\gamma)\varphi _1 (\delta)) =  \psi ( \varphi _1(\gamma \delta))
  =  \varphi _2(\gamma \delta) =  \varphi _2(\gamma)\varphi _2(\delta)
  =  \psi \varphi _1(\gamma)\psi \varphi _1(\delta) = \psi(u)\psi(v).
  \]
 \end{proof}

 The special case of interest to us is when $X$ is a the Riemann surface of the set of $\C$-points of an algebraic
 curve. In this case, if $\overline{X}$ is the canonical compactification of $X$, then the unramified coverings of $X$
 correspond uniquely to coverings $f: \overline{Y} \to \overline{X}$ branched over the points $B = \overline{X}- X$. The
 branching type at each point $b \in B$ is determined by the action of a generator of $\pi$ which ``goes once around $b$''.
 For instance, if $f$ is a Galois $G$-covering, then each $f^{-1}(b)$ can be identified to a coset $G/H$ for a subgroup
 $H$, the inertia at $b$, well-defined up to conjugacy. A loop around $b$ determines an element of $G$ whose action on the
 coset $G/H$ determines the branching type above $b$.

We apply this to the situation: $G = D_7 = \langle s, t \mid s^7 = t^2 = 1, s t = t s^6\rangle$, the dihedral group
with 14 elements. $X = \Proj ^1 (\C)- B$, where $B$ is a set of 6 points, thus
\[
\pi = \pi _1 (X, \ast) = \langle \gamma _1, ..., \gamma _6 \mid \prod _{i=1}^{6} \gamma _i = 1\rangle .
\]
We are interested in surjective homomorphisms $\rho :\pi \to D_7$ such that each $\rho(\gamma _i) $ is a nontrivial
involution (i.e., element of order 2). This corresponds to Galois $D_7$-coverings $Z \to \Proj ^1 (\C)$ branched
only above the six points in $B$ and with branching scheme $2,2,2,2,2,2,2$ above each point. Thus each
$\rho(\gamma _i) = s^{a_i}t $ for an integer $a_i$ modulo 7, and it is easy to see that $\rho$ is surjective if and
only if at least two of the $a_i$ are distinct. The condition $ \prod _{i=1}^{6} \gamma _i = 1$ translates into
$a_1 -a_2+a_3-a_4+a_5-a_6 =0$ modulo 7. By lemma (\ref{L:1}), two such $\rho$'s will define isomorphic coverings
if and only if they differ by an automorphism of $D_7$. The group $\mr{Aut}(D_7)$ has order 42, generated by the
substitutions
\[
(s, t)\mapsto (s^3, t), \quad (s, t)\mapsto (s, st).
\]
Modulo these automorphisms, the vector $(a_1, ..., a_6)$ is equivalent to any other
$b_1, ..., b_6$ where $b_i = c a_i + d$ where $c \in \FF _7 ^{\ast}$, $d \in \FF _7 $. We see that
the branched $D_7$-coverings we are considering are in one-to-one correspondence with the
set $(0, a_2, ..., a_6)$ modulo  7, with $(a_2, ..., a_6)\neq (0, ..., 0)$ such that $a_2-a_3+a_4-a_5+a_6 = 0$, all modulo
scaling $(a_2, ..., a_6)\mapsto (c a_2, ..., c a_6)$, $c \in \FF _7 ^{\ast}$, in other words to the hypersurface
in $\Proj ^4 (\FF _7)$ defined by this linear equation. We have shown:

\begin{proposition}
\label{P:1}
The isomorphism classes of Galois  $D_7$-coverings $Z$ of $\Proj ^1(\C)$ branched above
a set B of six given points with branching scheme $2,2,2,2,2,2,2$ above each branch point  is in a
noncanonical one-to-one correspondence with  $\Proj ^3 (\FF _7)$.
\end{proposition}
We can see this correspondence another way. These coverings have canonical structures as Riemann surfaces, and thus
as algebraic curves over $\C$. The quotient curve $Z/t = C$ has genus 2 and the projection $Z \to C$ is an unramified
cyclic covering of degree 7.  In fact, $C$ is the double covering of $\Proj ^1 (\C)$ branched above the 6 points in $B$.
By the geometric form of class field theory,  a cyclic degree 7 covering corresponds to a
subgroup of the Jacobian $H \subset \mr{Jac}(C)$ of order 7. Namely, since Jacobians are principally polarized, they are
self-dual. The dual of the isogeny $ \mr{Jac}(C) \to  \mr{Jac}(C)/H$ is a cyclic isogeny $A \to  \mr{Jac}(C)$ of degree 7, and
we get the covering $Z$ by pulling this isogeny back along the canonical embedding $C \to  \mr{Jac}(C)$. This curve has a
$D_7$-action:  we can choose the embedding $C \to  \mr{Jac}(C)$ so that  the map $P\mapsto -P$ on the Jacobian induces the hyperelliptic involution of $C$. Together $H$ and $P\mapsto -P$ define this $D_7$-action, since $H$ is also
preserved under $P\mapsto -P$. Since the
set of subgroups of order 7 is equal to the set of lines in $ \mr{Jac}(C)[7] $, the  points of order 7, and since this latter set is noncanonically
isomorphic to $\FF_7^4$, we get another parametrization of these coverings $Z$ by $\Proj ^3 (\FF _7)$. It would be interesting
to compare these two parametrizations.

\section{Construction of the curves.}
\label{S:constr}

In constructing curves $Y$ of genus 3 with endomorphisms by
$\QQ(\zeta_7 ^+)$, we can consider the curve $X$ which gives the
correspondence defining the endomorphism.  This curve $X$ has
genus 8.  There is a covering $X \to Y$ of degree 2 with 6 branch
points.  In fact $X$ has an action of the dihedral group  $D_7$,
symmetries of a regular 7-gon.  This is  a special case of
Ellenberg's construction. $D_7$ is of order 14 generated by two
elements   $s$ and $t$  with $s^7=1$, $t^2=1$, and  $t s t = s^6$.
Here is a diagram of the subgroups and corresponding curves and
function fields. $t_i = s^i t$ is one of the 7 involutions. All the curves
$Y_i$ are isomorphic.

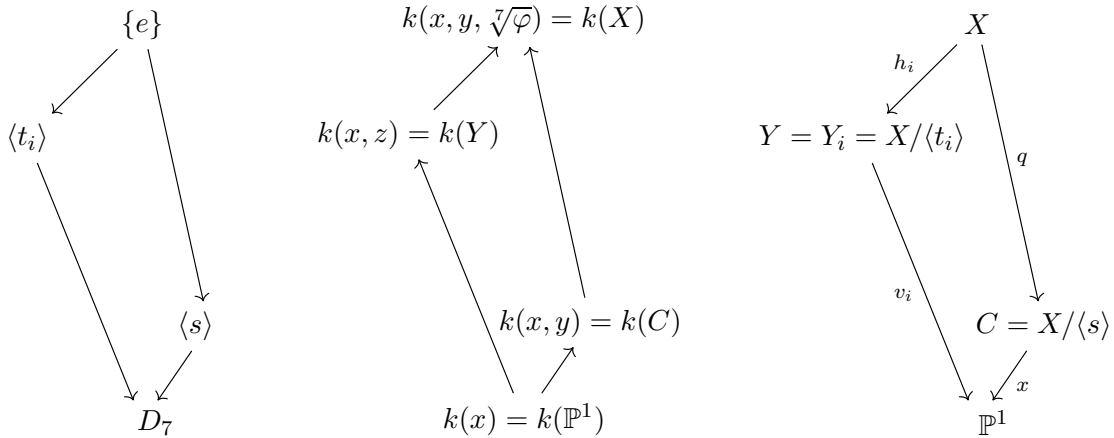
\[
\begin{tikzpicture}
  \node (P0) at (90:2.8cm) {$\{ e\}$};
  \node (P1) at (90+50: 2cm) {$\langle t _i  \rangle$} ;
  \node (P2) at (90+ 180: 2.5cm) {\makebox[5ex][r]{$D_7 $}};
  \node (P3) at (90+3*72:1.5cm) {\makebox[5ex][l]{$ \langle  s \rangle$}};

  \path[commutative diagrams/.cd, every arrow, every label]
    (P0) edge node[swap] {} (P1)
    (P1) edge node[swap] {} (P2)
    (P3) edge node{} (P2);
    \path[commutative diagrams/.cd, every arrow, every label]
     (P0) edge node {} (P3);
   \end{tikzpicture}\quad\quad
\begin{tikzpicture}
  \node (P0) at (90:2.8cm) {$k(x, y, \sqrt[7]{\varphi}) = k(X)$};
  \node (P1) at (90+50: 2cm) {$k(x, z)=k(Y)$} ;
  \node (P2) at (90+ 180: 2.5cm) {$ k(x) = k(\Proj^1)$};
  \node (P3) at (90+3*72:1.5cm) {$k(x, y) = k(C)$};

  \path[commutative diagrams/.cd, every arrow, every label]
    (P1) edge node[swap] {$$} (P0)
    (P2) edge node[swap] {$$} (P1)
    (P2) edge node{$$} (P3);
    \path[commutative diagrams/.cd, every arrow, every label]
     (P3) edge node {$$} (P0);
   \end{tikzpicture}   \quad\quad
  \begin{tikzpicture}
  \node (P0) at (90:2.8cm) {$X$};
  \node (P1) at (90+50: 2cm) {$Y = Y_i = X/\langle t_i \rangle$} ;
  \node (P2) at (90+ 180: 2.5cm) {\makebox[5ex][r]{$\Proj  ^{1} $}};
  \node (P3) at (90+3*72:1.5cm) {$C= X/\langle s \rangle$};

  \path[commutative diagrams/.cd, every arrow, every label]
    (P0) edge node[swap] {$h_i$} (P1)
    (P1) edge node[swap] {$v_i$} (P2)
    (P3) edge node{$x$} (P2);
    \path[commutative diagrams/.cd, every arrow, every label]
     (P0) edge node {$q$} (P3);
   \end{tikzpicture}
   \]

In this diagram:
\begin{itemize}
\item[1.] Genus of $X$ is 8.
\item [2.] The element $s \in D_7$ of order 7  acts fixed point free. $C = X/\langle s \rangle$, has genus 2, and is defined by an equation $y^2 = s(x)$ for a degree 6 polynomial with distinct roots.
$X \to C$ is unramified.
\item [3.] Any involution $t_i \in D_7$ has 6 fixed points on $X$. These lie above the six roots
of $s(x) = 0$. Each $Y_i = X/\langle t_i \rangle$ has genus 3.
\item[4.] Let $Y$ be any of the mutually isomorphic $Y_i $. The curve $X$ defines a correspondence of $Y$
which as an endomorphism of ${\rm Jac}(Y)$, satisfies the equation of $\zeta _7 ^+$, viz.,
$x^3+x^2 -2x -1 = 0$.
\end{itemize}

The element $\varphi (x) = a(x) + b(x)y \in k(C)$ is chosen in such a way that
the extension $k(X)/k(C)$, or equivalently the projection $X\to C$, is everywhere
unramified. Our approach to construct the curves $Y$ is to construct the curves $X$
with a $D_7$-action with certain properties. These will arise from a genus 2 curve
$C$ with an unramified cyclic 7-covering $X\to C$. Kummer theory (assuming $\zeta _7 \in k$) tell us that we may construct the 
extension of function fields  as $k(x,y,\sqrt[7]{\varphi})$. This will be unramified
if and only if ${\rm div} (\varphi) \in 7 {\rm Div}(C)$, i.e., every zero and pole has order a multiple of 7.
If $N : k(x, y)^* \to k(x)^*$ is the norm, this implies that $N(\varphi) \in 7{\rm Div}(\Proj ^1)$.
Clearly ${\rm deg (div}(N(\varphi)) = 0$, and every degree 0 divisor on $\Proj ^1$ is principal, so this shows
that
\begin{equation}
\label{E:constr1}
N(\varphi) = a(x)^2 - b(x)^2 y^2 = a(x)^2 - b(x)^2 s(x) = \lambda c(x)^7
\end{equation}
for some rational function $c(x)$ and constant $\lambda \in k$.

\begin{proposition}
\label{P:2}
Let $a, b, c \in k(x)$, $s \in k[x]$ of degree 6 with distinct roots, $\lambda \in k$.
Suppose that equation (\ref{E:constr1}) above is satisfied. Let $C$ be the genus 2 curve defined by
$y^2 = s(x)$. Assume that $\zeta _7 \in k$, and also that $\sqrt[7]{\lambda}\in k$.
\begin{itemize}
\item[1.] If $\varphi = a(x) + b(x) y$ is not a $7$th power in $k(x, y)$, then letting w be any root of
$Z^7 - \varphi $, the extension $k(x, y, w)/k(x)$ is Galois with group $D_7$.

\item[2.] If ${\rm div} (\varphi)$ is relatively prime to  ${\rm div} (\varphi ')$ where $\varphi ' =a(x)-b(x)y$, then
the extension $k(x, y, w)/k(x, y)$ is unramified. Thus if $X$ denotes the projective nonsingular model of
$k(x, y, w)$, then it satisfies the four properties listed after the diagram above.

\end{itemize}
\end{proposition}

\begin{proof}
To see the first point, note that $sw = \zeta _7 w$, defines an automorphism $s$ of the field
$k(x, y, w)$ fixing $k(x, y)$, which generates the Galois group of the degree 7 extension
$k(x, y, w)/k(x, y)$. Kummer theory tells us that all such cyclic degree 7 extensions are gotten this way.
The issue is to see that the hyperelliptic involution $(x, y)\mapsto (x, -y)$ lifts to an involution $t$
of $k(x,y,w)$ such that $ts = s^{-1}t$. Define $t$ by the rule
$tx =x, ty = -y$ and $t w = \mu c(x)/w$ where $\mu \in k$ is a root of $\mu ^7 = \lambda$. Since
\[
\left (\frac{\mu c(x)}{w}\right )^7 = \frac{N(\varphi)}{\varphi} = \varphi ',
\]
we see that $z = \mu c(x)/w$ is a root of the equation $X^7 - \varphi ' = 0$. As is well-known, this means that we can
lift the automorphism  $(x, y)\mapsto (x, -y)$ to an automorphism of $k(x, y, w)$ by sending $w\mapsto z$. The identities
$t^2 = 1, ts = s^{-1}t$ are immediate.

On the second point: any point $P\in C$ where the order of zero or pole $v_P(\varphi)$ is not divisible by 7 will
give rise to branching in the covering $X \to C$ above $P$. If $v_P(\varphi)$ is divisible by 7, there will be no branching.
From the equation $\pi ^* \mr{div}_ {\Proj^1} (N(\varphi)) = \mr{div}_C(\varphi)+ \mr{div}_C(\varphi ')$,
where $\pi : C \to \Proj ^1$, we see that as long as $\mr{div}_C(\varphi)$ and $\mr{div}_C(\varphi ')$ are relatively
prime, we can conclude that $\mr{div}_C(\varphi) \in 7 \mr{Div}(C)$ from the hypothesis that
$N(\varphi) $ is the seventh power of a rational function i.e., $\mu c(x)$.

$X$ has genus 8 since it is an unramified covering of degree 7 of a genus 2 curve. The other properties
about $s$ and $t$ are easy to check. We will prove the statement about the endomorphism of the Jacobian later.

\end{proof}

We will refine this result to show that we can define our curves by the above procedure by solving the equation
\begin{equation}
\label{E:basic}
a(x)^2 - s(x)b(x)^2 = c(x) ^7
\end{equation}
for polynomials $a, b, c, s$ of respective degrees $7, 4, 2, 6$. For simplicity, we ignore the constant $\lambda$, which can implicitly be absorbed
into the equation if our field $k$ is algebraically closed. The systematic study of this Diophantine equation appears in section \ref{S:polys}.
We consider genus 8 curves $X$ with an action of $D_7$ such that $X/D_7$ has genus 0, the generator $s$ has no fixed point on $X$ and
each involution $t$ has six fixed points on $X$. We will show:

\begin{proposition}
\label{P:3}
Let $(z, x, y)$ be coordinates in projective space $\PP^2$. We
let $D_7$ act on $\PP^2$ by the formulas:
\[    s(z)=z, \, \,  s(x) =
\zeta_7 x, \, \, s(y) = \zeta_7^{-1} y,  \quad t(z)=-z, \, \, t(x)
= y, \, \, t(y) =x. \]

A genus $8$ curve $X$ with a  $D_7$-action with the above properties
has an equation
in the shape:
\[ (x^{14} +y^{14}) +   \phi(xy, z^2) +z(x^7-y^7) \psi(xy, z^2)=0,\]
where  each term in   $\phi(xy, z^2)$   has total degree $14$, and
each term in $\psi(xy, z^2)$  has degree $6$. The above equation is invariant under $D_7$.
In affine coordinates $x = x/z, \, y=y/z$,
$D_7$ acts as
 \[      s(x) =
\zeta_7 x, \, \, s(y) = \zeta_7^{-1} y,  \quad  t(x) = -y,  \, \,
t(y) = -x. \] With an obvious change in notation, the equation
becomes
\[         (x ^{14} +y^{14}) + \phi(xy) +  (x^7-y^7)
\psi(xy) =0,\] where  $\phi(w)$ has degree $7$ and  $\psi(w)$  has
degree $3$.Thus a genus $8$ curve $X$ with a  $D_7$-action with the above properties has an affine plane model in the shape
\begin{equation*}
\label{E:D7}
\begin{matrix} u^2+v^2 +   \phi( w) + \tau \pi(w) =0, \\
\text{where $u = x^7, \, \,  v = y^7,   \, \, w
= -xy, \, \, \tau = u-v , \, \, \deg(\phi(w))=7, \, \,
\deg(\pi(w))=3$.}
\end{matrix}
\end{equation*}

\end{proposition}
The proof of this proposition involves the following steps: We consider the canonical embedding
$$X \to \PP^7 = \PP(H^0 (\Omega _X)), $$
which is $D_7$-equivariant for a linear action of $D_7$ on the module of differentials $H^0 (\Omega _X)$. We determine
this action, see propositions \ref{P:7}, \ref{P:8} below. The planar equation in the above proposition is the image of the degree 14 canonical
curve under an equivariant projection $ \PP(H^0 (\Omega _X))\to \PP^2$ with the $D_7$-action in this proposition. The equation
displayed is the most general degree 14 polynomial invariant under this $D_7$-action. Note that further conditions on the coefficients
of $\phi, \psi$ must be satisfied in order that the displayed equation defines a genus 8 curve of the type we are considering. We examine these conditions
next. We consider curves defined by an equation (recall: $w = xy,   \, \,   u = x^7,   \, \,   v =  y^7,    \, \,  \tau =
u-v; \, \,  \Rightarrow   \, \, u^2+v^2 =\tau^2+2 w^7$)
$$f(x,y)=u^2+v^2+\phi(w) +\tau \psi(w) =g(\tau,  w)=0, \, \, \text{ where } \, \, g(\tau,  w) = \tau^2 +\tau \psi(w)  + ( \phi(w) +2 w^7). $$
Note that the expression on the right  quadratic in $\tau$.   The terms $w$ and $\tau$
are invariant under $D_7$. The polynomials $\phi, \psi$ have respective degrees 7, 3. We let $\FF$ be a field containing $\QQ$ and all the coefficients
of the polynomials $\phi, \psi$. Initially we treat these coefficients as independent variables, to be specialized  later. In the statements below, generic
assumptions are made so that the statements are meaningful, (e.g., that the equations are irreducible and hence define field extensions). This does not
involve a loss of generality for our purposes. Our goal is to define families of curves with general moduli. It will be justified later that we really do get
families of maximal modular dimension.

Let $C_7=\langle s \rangle$ be the subgroup generated by $s$,  and
let $I=\langle t \rangle$ be the subgroup generated by $t$. The
function field of the curve $X$ is $K = \FF(x, y)$ where $f(x, y)
= 0$. If $H$ is a subgroup of $D_7$, we let $K^H$ be the subfield
of elements if $K$ fixed by $H$. This is the function field of the
quotient curve $X/H$, note that $K/K^H$ is a Galois extension with
group $H$.

We also consider the fields $\FF(\tau, w)$ and $\FF(u, v, w) =
\FF(u, w) = \FF(v, w)$, which are subfields of $K$.

\begin{proposition}
\label{P:4}
\[ \FF(u, v, w) = K^{C_7},   \quad \text{and}  \quad  \FF(\tau, w) =
K^{D_7}.\]
\end{proposition}
\begin{proof} Since we have seen that   $\FF(u, v, w) \subset K^{C_7}$,
and $\FF(\tau, w) \subset K^{D_7}$.  From the equation $f(x, y) =
0$, we see that both $x$ and $y$ satisfy an equation of degree 7
over $\FF(u, v, w)$, and since  $xy = w$, we see that $\FF(x, y) =
\FF(u, v, w)$, so that $K/ \FF(u, v, w)$ has degree at most 7. On
the other hand $K/K^{C_7}$ has exact degree 7 so $\FF(u, v, w) =
K^{C_7}$.

The equation \[h(T) =  (T-u)(T+v) = T^2 - \tau T -w^7 = 0\]
shows that $\FF(u, v, w)$  is a quadratic extension of $\FF(\tau,
w)$. Therefore $\FF(x, y)/\FF(\tau, w)$ has degree at most 14 and
since $K/K^{D_7}$ has exact degree 14, we have $\FF(\tau, w) =
K^{D_7}$.
\end{proof}

Since $\FF(\tau, w) = K^{D_7} = \text{function field of } X/D_7$,
to require that $X/D_7$ has genus $0 (= \PP^1)$ we get

\begin{proposition} \label{P:g0}   $X/D_7$ has genus $0$ if and only if
\[          \psi(w)^2  - 4 ( \phi( w) +2 w^7) =   L(w) C(w)^2,
\quad \text{ where } \quad \deg L(w)\leq 1,  \, \,  \deg C(w)\le
3. \]
\end{proposition}

\begin{proof}  $\FF(\tau, w)$ is a quadratic extension of $\FF(w)$
with equation $g(\tau, w) = 0$.  $\FF(w)$ has genus 0.  Then
$\FF(\tau, w)$ will have genus 0 if and only if the extension $\FF(\tau,
w)/\FF(w)$ has at most 2 ramification points since by the
genus formula.  But this extension is $\FF(w, d)$ where
$d^2 =\text{discriminant of }g$. The expression for the
discriminant of $g$ is on the left-hand side above.  It will have
at most 2 ramification points on $\PP^1_w$ if and only if it has the shape
above (note that the right-hand side has degree at most 7).
\end{proof}

The most important case is where $L(w)$ really is a linear
expression,  say $L(w) = w-a$.  We will assume $a \neq 0$.  We let
$m^2 = w-a$. Then under the hypotheses of Proposition \ref{P:g0}
we have $\FF(\tau, w) = \FF(m)$.  Solving $g(\tau, w) = 0$ by the
quadratic equation we get
\[        \tau (m)  = \dfrac{ -   \psi(  m^2+a  ) \pm m C(m^2+a) }{2}, \quad \text{and define} \quad
 \tau (m) : = \dfrac{ -   \psi(  m^2+a  ) + m C(m^2+a) }{2}.  \]

Note  that, when $C(w)$ and $\psi$ are arbitrary cubic polynomials in
$w$, the expression for $\tau(m)$ represents the general
polynomial of degree 7. Given any degree 7 polynomial $\tau(m)$,
if $a$ and $C(w)$ are given, we can solve the above expression for
$\psi$ and then solve for $\phi$ from the expression in
Proposition \ref{P:g0}; this gives the equation for our curve
$f(x, y)$.

Next we need to see when  the curve $X/C_7$ has genus 2. The
extension of function fields of $X/C_7$ over the function field of
$X/D_7$, which is  $\FF(m)$ by our assumption, is a quadratic
extension, with equation $h(T)$ as in the proof of Proposition
\ref{P:4}. By the well-known genus formula, this will have genus
2 if and only if there are exactly 6 ramification points in the covering
$X/C_7 \to X/D_7=\PP^1_m$. The ramification is given by the
discriminant of $h(T)= T^2 - \tau T -w^7 = 0$, which is
$\tau(m)^2+4w^7 = \tau(m)^2+4(m^2+a)^7$. Therefore we obtain:

\begin{proposition}
\label{P:5}  Under the hypotheses of Proposition $\ref{P:g0}$,
$X/C_7$ has genus $2$ if and only if
\[    \tau(m)^2+4(m^2+a)^7 = q(m)^2 s(m), \, \,  \text{
where  } \, \, \deg q(m) =4, \, \, \deg s(m) = 6.\]
\end{proposition}
Note that this is in the shape of equation \ref{E:basic}, namely after absorbing constants, 
$a(x)^2 - s(x)b(x) = c(x)^7$.

\noindent{\bf The procedure to write down our genus 8 curves: }
\begin{itemize}
\item[1. ] Choose an  $a \neq 0$; and choose polynomials $q(m)$
and $s(m)$ of degree 4 and 6 respectively.

\item[2. ]    Find a solution $\tau(m)$ of degree 7 to the
equation
\[    \tau(m)^2+4(m^2+a)^7 = q(m)^2 s(m), \, \, \text{
where  } \, \, \deg \tau(m) =7, \, \, \deg q(m) =4, \, \, \deg s(m) =
6.\]

\item[3. ]  Choose a polynomial $C(w)$ of degree 3. Solve  for
$\psi$ to  the equation:
\[      \tau (m) : = \dfrac{ -   \psi(  m^2+a  ) + m C(m^2+a) }{2}.
\]

\item[4. ] Let $L(w) = w-a$. Solve  for  $\phi$ in the equation:
\[    \psi(w)^2  - 4 ( \phi( w) +2 w^7) = L(w)
C(w)^2.
\]

\item[5. ]  Obtain genus 8 curve of the following form:
\[  g(\tau,  w) = \tau^2 +\tau \psi(w)  + ( \phi(w) +2 w^7)=0.\]
\end{itemize}

 To construct the genus 3 curves, we take the quotient $X/t$
 of the genus 8 curves $X$  by any involution $t^2=1$.  There are
7 involutions, all conjugate, so all these genus 3 curves will be
isomorphic.

\begin{proposition}
\label{P:6}
Let $t$ be  the involution  $tx = -y, \,
ty=-x$.  Let $r = x-y$, and  $I$  the subgroup of $D_7$ generated
by $t$.   Then $K^I = \FF(r, w)$, where  $K^I$ is the function field
of $X/I$, and $K = \FF(x, y)$.
\end{proposition}

\begin{proof}
By Proposition \ref{P:4}, $K^{D_7} = \FF(\tau, w)$, where $\tau
= x^7-y^7$; and also $K/K^{D_7}$ is a Galois extension with group
$D_7$, therefore of degree 14.   Clearly  $\FF(\tau, w)$ is a
subfield of $K^I$. On the other hand, $\FF(r, w)$ is an extension
of $\FF(\tau, w)$ of degree 7.  In fact, $r=x-y$ satisfies the
equation
\[       h(r,w)  =      r^7 + 7 w r (r^2 + w)^2 - \tau = 0 .\]
This shows that $\FF(x, y)$ is a degree 2 extension of $\FF(r, w)$
and therefore $\FF(r, w) = K^I$.
\end{proof}

The equation $h(r, w) = 0$ is an equation for the curve of
$X/I$.  We want this curve to have genus 3.  Recall in Proposition
\ref{P:g0}, in order to have $K^{D_7}$ to be of genus 0, the
conditions are imposed on $\tau$ and $w$ that is to express
$\tau = \tau(m)$ with $m^2 = w-a$.   Since $X \to  X/{D_7}=
{\PP^1}$ is branched above 6 points which these are the fixed
points of all the involutions in $D_7$,  we must have $X/I \to
X/{D_7}= {\PP^1}$  will also be branched above 6 points, but this
is a non-Galois extension.  The branching type will be $2,2,2,1$
over each of the 6 points.  Such a curve will have genus 3.  In
order to have this branching behavior, the curve $h(r, w) =
0$, must intersect the discriminant of the extension $\QQ(r,
w)/\QQ(\tau, w)$ transversally in 6 points.

Consider $r, w$  and $\tau, w$ are independent variables so that
the extension of fields $\QQ(r, w)/\QQ(\tau, w)$ represents a
covering of degree 7 of planes, either affine planes  $\A^2 \to
\A^2$, or projective planes $\PP^2 \to \PP^2$.  The discriminant
of the polynomial $h(r, w)$ with respect to $r$ is
\[\Delta_{h} =  -7^7 (\tau^2+4 w^7)^3.\]
Notice that $\tau^2+4 w^7$ was exactly the discriminant we
calculated before.  The condition we need is that, as a function
of $m$, the polynomial $\tau^2 +\tau \psi(w)  + ( \phi(w) +2
w^7)$, which is the equation of the quotient curve $X/D_7$,
intersects the discriminant locus in 6 transversal points
(tangential intersections do not necessarily give ramification).

In summary: to construct our curves, we must find polynomial solutions to the equation
\[ a(x)^2-s(x)b(x)^2=c(x)^7, \quad \text{where } \deg a=7, \,
\deg s= 6, \, \deg b =4, \, \deg c=2.\] 
There is a 
total of $ 8+(7+5-1)+3 =22$ variables which are coefficients of the
polynomials of $a, b, c, s$, and total of 15 polynomial equations in those unknown coefficients which
result by comparing like terms in the above equation. Thus, the solution set should be
$22-15=7$ dimensional. Here is a procedure to solve the above Diophantine equation.
Without loss of generality, we homogenize the above equation with coordinates $X, Z$, and assuming $c(x)$ has two 
distinct roots, we can make a linear change in $X, Z$ so that 
$c(x) = XZ$ and rewrite the above
equation as
\[ S_7^2-(XZ)^7= F_6  \prod_{i=1}^4(X-u_i^2Z)^2,    \quad \text{where } \deg S_7=7, \,
 \deg F_6 =6, s(X) = F_6(X, 1), \, u_i^2 \text{ are the distinct
roots}.\] Set $Z=1$, and let $s_7(X)=S_7(X, 1)=\sum_{i=0}^7a_i
X_i$, we obtain
\[  s_7(X)^2- X^7= F_6  \prod_{i=1}^4(X-u_i^2 ) ^2.\]
Hence, we have the following 8 equations gotten by putting in $X = u_i ^2$ into the above and its derivative:
\begin{eqnarray*}
s_7(u_i^2)^2 = u_i^{14},   & \text{ and } & 
2s_7(u_i^2) s_7'(u_i^2) - 7 u_i^{12} =0, \quad i =1, 2, 3, 4.
\end{eqnarray*} 
We can factor the first as $s_7(u_i^2) = \pm u_i^{7}$. Choosing the plus sign and putting into the second 
equation yields
\begin{eqnarray*}
s_7(u_i^2) = u_i^{7},   & \text{ and } & 
2 s_7'(u_i^2) - 7 u_i^{5} =0, \quad i =1, 2, 3, 4.
\end{eqnarray*} 
Now this is a system of 8 {\it linear} equations in the 8 unknown coefficients of $s_7$. Using Cramer's 
rule, we find that the coefficients are explicit rational functions of the variables $u_i$, $i = 1, 2, 3, 4$.
Once we obtain $s_7(X)$,
we divide $s_7(X)^2-X^7$ with $ \prod_{i=1}^4(X-u_i^2 ) ^2$ to determine $s(X)=F_6(X,1)$,
and then we  let $y^2=F_6(X,1)$,which is a genus 2 curve. Further analysis of this equation can be found in section \ref{S:polys}.

\begin{theorem}
\label{T:1}
\begin{itemize}
\item[1.] There is a 4-parameter family of solutions to the equation 
\[
a_u(x)^2 -s_u(x)b_u(x)^2 = x^7
\]
for polynomials $a_u(x), b_u(x), s_u(x)$ of respective degrees 7, 4, 6. The coefficients of these polynomials are in the 
field $\QQ (u_1, u_2, u_3, u_4)$, for variables $u_1, u_2, u_3, u_4$. 

\item[2.] For $u = (u_1, u_2, u_3, u_4)$ in a Zariski-dense open subset of $\mathbb{A}^4$, 
$y^2 = s_u(x)$ defines a genus 2 curve $C_u$ over the field $\QQ (u)$. Letting 
$\varphi = a_u(x) + b_u(x)y$, $\QQ (u, x, y, \sqrt[7]{\varphi})$ is the function field of a genus 8 
curve $X_u$ defined over $\QQ (u)$.  This curve has an action of the group $D_7$, defined over $\QQ (\zeta _7, u)$.
Dividing this genus 8 curve by the action of any involution in $D_7$ we obtain a genus 3 curve 
$Y_u$, defined over $\QQ (u)$, such that $\mr{End}(\mr{Jac}(Y_u))$  contains $\ZZ[\zeta _7 ^+]$. The endomorphisms are
defined over $\QQ (\zeta _7, u)$. 

\end{itemize}

\end{theorem}
Explicit equations for the curves $Y_u$ can be found in the appendix, section \ref{S:explicit}. In fact, the equations are symmetric 
functions of $u = (u_1, u_2, u_3, u_4)$, and so can be expressed in terms of elementary symmetric functions.

\section{Representations of $D_7$.}
\label{S:reps}
Here we explain the where Proposition \ref{P:3} comes from, by computing the relevant actions of $D_7$.
If a finite group $G$ acts on a curve $X$ we get an action of $G$ on the cohomology of the curve.
The action of $G$ on $H^0(X)$, $H^2(X)$ will be trivial. Also, by the Hodge decomposition:
\[H^1(X, \mathbb{C}) = H^0(X, \Omega_X) \oplus \overline{H^0(X, \Omega_X)} \quad \text{with}
\quad  \overline{H^0(X, \Omega_X)} = H^1(X,\mathcal{O}_X) ,\]
the representation of $G$ on $H^1 (X, \mathbb{C})$
decomposes as $r+\bar{r}$ , where $r$ is the representation on the
$g$-dimensional space of holomorphic differential 1-forms.
Recall the Lefschetz fixed point formula:
\[       (h^0-h^1+h^2) (u)  = \text{fix}(u),  \quad \forall  \, u \in G,\]
where  $h^i$  is the character of the $G$-module  $H^i$, and $\text{fix}(u)$ is the number of fixed
points of $u$ on the curve $X$,  counted with multiplicity. These characters depend only on the conjugacy class of $u$ in $G$.

  In our case $G=D_7$, $g = 8$. Also we know the fixed points because we know the ramification data in the various coverings
  $X \to X/H$  for subgroups $H$.  We know
\begin{itemize}
  \item[1.]   $X \to X/t$  has 6 branch points because the quotient has genus 3.
(Recall the genus formula: If  $X \to Y$  is a covering of curves of
degree $n$, then  $n(2-2g_Y) -e  = 2-2g_X$ where $e$ is the total
ramification (branching) order of the covering.  Since $g_X=8$ and
$g_{X/t}=3$, the genus formula $2(2-2(3))-e=2-2(8)$ shows that
$e=6$.)

 \item[2.]  $X \to X/s^\alpha$  has 0 branch points because the quotient has genus 2,
 $\alpha=1, \ldots, 6$.  (Since $g_X=8$ and $g_{X/s^\alpha}=2$, the genus formula  $7(2-2(2))-e=2-2(8)$ shows that
 $e=0$.)
\end{itemize}
So $\text{fix}(t) = 6$ and $\text{fix}(s^\alpha) = 0$. There are 5
conjugacy classes in $D_7$ and therefore 5 isomorphism classes of
irreducible representations:
\begin{itemize}
\item[1.] The simple one dimensional representations are given by
$\sigma=1$ and $\tau =\pm 1$.  Let us denote them by $1$ and
$\alt$.  \item[2.] The simple two dimensional representations are
given by $\sigma = \left(
\begin{matrix} \zeta^k & 0 \\ 0 & \zeta^{-k} \end{matrix} \right)$, and $\tau =
\left( \begin{matrix}  0  & 1 \\ 1 & 0 \end{matrix}  \right)$
where $\zeta = e^{2 \pi i/7}$,  $1 \leq  k \leq 3$.  Let denote
these by $\chi_1, \chi_2, \chi_3$.
\end{itemize}

The conjugacy classes are represented by $1,   s, s^2, s^3, t$
(that is the classes of $1, \{s, s^6\}, \{s^2, s^5\}, \{s^3,
s^4\}, t$), and the irreducible representations are denoted by $1,
\alt, \chi_a$ for $a=1,2,3$ respectively. The character table with
columns indexed by the conjugacy classes, the rows indexed by the
irreducible representations, and the entries are the values of the
characters:
\begin{center}
\begin{tabular}{|l|l|l|l|}  \hline
& 1 &  $t$   &  $s^b$ (\text{where } $b=1, 2, 3$) \\
\hline
1 & 1 & 1 & 1  \\
\hline
$\alt$   \,  & 1 & -1 & 1 \\
\hline  $\chi_a$  (\text{where } $a=1, 2, 3$) & 2 & 0 &  $\zeta^{ab}+\zeta^{-ab}$  \, (\text{where } $\zeta=e^{2 \pi i/7}$)   \\
\hline
\end{tabular}
\end{center}
 Let $\alpha$ be the character $\chi_1+\chi_2 +\chi_3$, which is defined over $\QQ$.
 In fact, $\alpha(1)=2+2+2=6$, $\alpha(t)=0+0+0=0$,
 $\alpha(s^b)= -1$ via the identity $\sum_{i=1}^6 \zeta^i =-1$.

\begin{proposition}
\label{P:7}
The character $r$ of  $D_7$  acting on the $8$ dimensional vector
space $H^0 (\Omega_X )$ is  $2 \cdot \alt +  \alpha$.
\end{proposition}

\begin{proof}
From  Lefschetz fixed point formula:
\[       (h^0-h^1+h^2) (u)  =(2-h^1)(u)= \text{fix}(u),  \quad \forall  \, u \in G, \text{ where $\text{fix}(u)$ is the number of fixed point}.\]
Since $\text{fix}(t) = 6$ and $\text{fix}(s^\alpha) = 0$, we have
$h^1(t) = -4$,  $h^1(s^\alpha)=2$, and $h^1 (1)= 16$. Comparing
with the  above characteristic table  $h^1=4\cdot \alt+2 \cdot
\alpha$, and hence $h^1= r+\bar{r}$, and the claim follows
directly.
\end{proof}

A model of the representation $\alpha$ is the following:
We let $V$ be the 6-dimensional $\QQ$-vector space which is the subspace of the cycltomic field $\QQ(\zeta)$ consisting of
         \[      a_0+a_1  \zeta +a_2 \zeta^2+a_3 \zeta^3+a_4 \zeta^4+a_5 \zeta^5+a_6 \zeta^6, \quad \text{such that } \, \sum_{i=0}^6 a_i =0. \]
Then $s$  acts on $V$ by multiplication by $\zeta$  and $t$  acts on $V$ by complex conjugation.

Recall that every smooth projective non-hyperelliptic curve admits a canonical embedding
                        \[ X  \to \PP (H^0 (\Omega_X )) = \PP^{g-1}.\]
This is $G$-equivariant if a group acts. In our case $g = 8$, and the action of $D_7$ is via the linear representation  $r$  above. One strategy to write down explicit curves of genus 8 with a  $D_7$-action is to write down the ideal $I$ of the canonical embedding.   In our case, it turns out that $I$ is generated by 15 quadrics with 35 syzygies. In fact, the entire free graded resolution was worked out by F. O. Schreyer (\cite{CEFS}).

We can construct an equivariant projection of the curve $X$ via
\[       \PP (2 \cdot \alt \oplus \alpha) = \PP^7  \to \PP (\alt \oplus \chi_1) = \PP^2.\]
This means that we represent $X$ by one equation which is $D_7$-invariant.  It has degree $14=2g-2$.
 The group $D_7$ acts on coordinates $(z, x, y)$  by  $s(z) = z$, $s(x) =\zeta x$, $s(y) =\zeta^{-1} y$, $t(z) = -z$, $t(x) = y$, $t(y) = x$,  $\zeta= e^{2 \pi  i/7}$.

We are interested in two things

\begin{itemize}
\item[1.] The equation for the projection of $X$. This will be a
polynomial $f(x, y, z) = 0$ of degree 14 and $D_7$-invariant. This
curve will necessarily have 70 singularities, because $X$ has
genus 8 (for the generic case, there will have 70 ordinary double
points by the genus and degree formula $g=(d-1)(d-1)/2 -
\sum_{i=1}^k r(r-1)/2$ where $d$ is the degree of the curve, $r$
is the order of the ordinary singularity, and $k$ is the number of
singular points). The location of these double points is important
and not arbitrary. In fact, the singularities will have the
following structure, in the ``generic case'':
   \begin{itemize}
   \item[a.]  Three $D_7$-orbits size 14,  belonging to non-fixed points of   $D_7$.
  \item[b.] Four orbits of  size 7 belonging to $t$-fixed points.
\end{itemize}
\item[2.] The adjoint curves of degree $d-3 = 11$, passing through all the double points of $X$. These form an 8-dimensional vector space isomorphic to the module of differentials $H^0 (\Omega_X)$.
     \end{itemize}

   These adjoints must form a representation of $D_7$   isomorphic to $2 \cdot \alt + \alpha$, by
   Proposition \ref{P:7}.  This condition limits the possibilities for the singularities of $X$.

\begin{lemma}
\label{L:2}
The fixed points of $D_7$ in $\PP(\alt \oplus \chi_1)$ are:
\begin{center}
\begin{tabular}{|l|l|l|}  \hline
\rm{Fixed point} & \rm{Stabilizer subgroup} & 
\rm{Size of orbit}  \\
\hline
$(1,0,0)$ & $D_7$  & $1$   \\
\hline
$(0,1,0)$ & $\langle s \rangle$  & $2$   \\
\hline
$(0,0,1)$ & $\langle s \rangle$  & $2$   \\
\hline
$(1,x,-x)$ & $\langle t \rangle $  & $7$   \\
\hline
$(0, 1 \pm 1)$ & $\langle t \rangle $  & $7$   \\
\hline
\end{tabular}
\end{center}
The fixed points $(1,\zeta^a x, \zeta^{-a} x)$ of  $s^a  t s^{-a}$  are in the orbit of $(1, x, -x)$.
The fixed points $(0, \zeta^a,  \pm \zeta^{-a})$ of  $s^a  t s^{-a}$  are in the orbit of $(0, 1, \pm 1)$.  We have a line of $t$-fixed points $x+y = 0$.
\end{lemma}
\begin{proof}
This can be checked case by case.
\end{proof}

Let $V$ be the representation $\alt \oplus \chi_1$ of $D_7$. We
want to study the action of $D_7$ on polynomials in $V$, in other
words the symmetric powers $\Sym^n (V)$. We need to know the
multiplication (i.e., the tensor products) of the basic
representations.

\begin{lemma}
\label{L:3}
We have
\[ \Sym^n  (\alt+\chi_1 )= \sum_{i=0}^n \alt^i  \otimes \Sym^{n-i} (\chi_1).\]
Also  $\alt^i=1$, and $\alt$ depending on whether $i$  is even or odd, and $\alt \otimes \chi_a = \chi_a$.
\end{lemma}
\begin{proof}  The claim is true since
it is know that if $\{\bv_1, \ldots, \bv_m\}$ is a basis for $V$,
then a basis for $\Sym^n(V)$ is $\left\{ \dfrac{1}{n!}
\sum_{\sigma \in S_n} \bv_{k_{\sigma(1)}} \otimes \cdots \otimes
\bv_{k_{\sigma(n)}} \right\}$ as $1 \leq k_1 \leq \cdots \leq k_n
\leq m.$
\end{proof}

\begin{remark}
\label{R:1}
The character of $\Sym^m(V)$ is the linear combination of $\{
\prod_i \chi(g^{r_i})^{e_i} ~|~  \sum_i r_i e_i =m \}$ where $g
\in G$.
\end{remark}

It remains to calculate the $\Sym^i(\chi_1)$. The easiest way to do this is to consider the restriction to the cyclic subgroup
$\langle s \rangle$.  Alternatively, one can use Molien's formula. The result is:

\begin{proposition}
\label{P:8}
We have
\[      \Sym^{11}  (V)= 3 \cdot 1+9 \cdot \alt+11 \cdot \alpha \quad \text{and} \quad
\Sym^{14}  (V)= 13 \cdot 1+5 \cdot \alt+17 \cdot \alpha.\]
\end{proposition}

This shows that the space of  $D_7$-invariant polynomials of
degree 14 is 13 dimensional. This is spanned by $z^{2i}
(xy)^{7-i}$ for $i = 0, \ldots, 7$;  $z(x^7-y^7 ) z^{2i}
(xy)^{3-i}$ for  $i=0, \ldots, 3$; and  $x^{14}+y^{14}$.  A linear
combination of these gives the equation $f(x, y, z) = 0$ as
claimed in proposition \ref{P:3}.

We have an equivariant exact sequence  
 \[  \begin{CD}
     0@>>>   H^0 (\Omega_X )  @>\lambda>>  \Sym^{11}  (V)@>\mu>> 
     \displaystyle{\bigoplus _{P\in \rm {Sing}(X)}} \ \C @>>> 0
\end{CD}        
\]
where the map $\mu$ sends each degree 11 homogeneous polynomial to its value at each of the 
$70$ singular points of the curve $X$. We making a general position assumption that $X$ has $70$ 
ordinary double points. An element in $ \Ker(\lambda)$ is an adjoint curve and defines a regular 
differential on $X$ in the usual way. Since we know the character of the $D_7$-action of first two terms above, 
we obtain the character of the action on the right-hand side, namely

\[     \displaystyle{\bigoplus _{P\in \rm {Sing}(X)}} \ \C  = (3 \cdot 1+9 \cdot \alt+11 \cdot \alpha ) -(2 \cdot \alt +\alpha) =  3 \cdot 1 + 7 \cdot \alt + 10 \cdot \alpha. \]
Since
the curve is invariant, the singular set of it will lie in orbits
(of size 1, 2, 7, or 14). The sum must form a
representation isomorphic to  $3 \cdot 1 + 7 \cdot \alt + 10 \cdot
\alpha$. This puts constraints on the location and types of orbits.

Each of these orbits defines an induced representation of the stabilizer of one of its points. We analyze 
these orbits now. If $P$ is a point in the projective plane which is fixed
by a subgroup $H \subseteq  D_7$, the linear form ``evaluation of
$g$ on the orbit defined by $P$''  is the induced representation
\[  \text{Ind}_H^{D_7} (\vartheta), \quad \text{ where $\vartheta$ is the character of  the group $H$, ``evaluation of $g$ at $P$''.}\]

\begin{lemma}
\label{L:4}
$\mr{Ind}_1^{D_7}(1)= 1 + \alt + 2 \alpha$, and $\mr{Ind}_{\langle t \rangle}^{D_7}(\mr{sgn})=\alt + \alpha$, where
$\mr{sgn}$ is the character of the group $\langle t \rangle$, with $\mr{sgn}(t) = -1$.
\end{lemma}
\begin{proof}
The first is the well-known decomposition of the regular
representation: it is a sum of all the irreducible representations, each appearing with
multiplicity equal to its degree. The second follows easily from Frobenius reciprocity.
\end{proof}

If $P$ is a non-fixed point, we have $H=1$, so the evaluation on the orbit of a non-fixed point gives a contribution $= 1+\alt+2\alpha$  to $\Ker(\lambda)$.  Evaluation on a $t$-fixed point of the type $(1, x, -x)$ gives the character $\text{sgn}(t) = -1$, because
$g(tP) = g(-1, -x, x) = -g(1, x, -x)$, since degree $g = 11$ is odd. Therefore we get a contribution  $\alt+ \alpha$ to $\Ker(\lambda)$. Since
\[    3 (1+\alt+2 \alpha) + 4(\alt+ \alpha) = 3 \cdot 1 + 7 \cdot \alt + 10 \cdot \alpha,\]
this suggests that our curve $X$ should most likely have 70 singularities distributed in three sets of 14 which are non-fixed points ,
and four sets of $t$-fixed points on the line $x+y=0$ (orbits of size 7). If the equation of the curve is $f(x, y) = 0$, the condition that
$f(x) := f(x, -x)$ shall have 4 double roots means that $ f(x)= q(x)^2 s(x)$ where  $\deg q = 4, \, \deg s = 6$. On the other hand,
we must also have $f(x)= h(x)^2+g(x)^7$ by reasoning that proved proposition \ref{P:5}. We are led again to an
equation
\[     q(x)^2 s(x)=h(x)^2+g(x)^7, \quad \text{ where  }\,  \deg q = 4, \, \deg s = 6, \, \deg h = 7, \,  \deg g = 2. \]


\section{On the moduli of this family}
\label{S:moduli}

Recall that our genus three curves are constructed as follows: One starts from suitable polynomials
$f(x), p(x), q(x), h(x)$ of respective degrees $6, 7, 4, 2$ such that $p^2-fq^2 = h^7$. Let $C$ be the genus 2 curve
with equation $y^2 = f(x)$, and $X$ the genus 8 curve that is the cyclic unramified covering of $C$ of degree 7 given
by adjoining a 7th root of $p(x) + q(x)y$, which is in the function field of $C$. The function field of $X$ is $k(x, y, z)$,
where $z^7 = p(x) + q(x)y$. Multiplying this by  $z^7 = p(x) - q(x)y$ we get an equation
$$z^{14} -2p z^7 +h^7=0,$$
which defines a Galois $D_7$ covering of $k(x) = k(\Proj^1)$, with generators $\sigma : x\mapsto x, y\mapsto y, z\mapsto \zeta _7 z$, and
$\tau: x\mapsto x, y\mapsto -y, z \mapsto h/z$. The subfield of $k(x, y, z)$ invariant under the involution $\tau$ is
$k(x, y, w)$, where $w = z+ h/z$, and it is the function field of a genus 3 curve $Y$ with an action of $\ZZ[\zeta _7 ^+]$ on its Jacobian. The equation of $w$ is
\[
w^7 -7hw^5+14h^2z^3-7h^3w -2p=0.
\]
The discriminant of this equation  with respect to $w$ is $-2^67^7 (fq^2)^3$, which shows that the ramification of the covering
$k(x, y, w)/k(x)$ is in the set $fq=0$; in fact it is ramified exactly in the 6 points of $f=0$.

 In this manner, we have a map of moduli spaces $\lambda: \mfr{M}_2 (7) \to \mfr{M}_3$, where $\mfr{M}_2 (7) $
is the moduli space of $(C, G)$ where $C$ is a genus 2 curves and $G \subset \mr{Jac}(C)$ is a subgroup of order 7, and
$\mfr{M}_3$ is the moduli space of genus 3 curves. Since $\dim \mfr{M}_2 (7) =3$ our goal is to prove that
the image of this map is 3-dimensional. We do this by showing that the map on tangent spaces
\[
T_a \mfr{M}_2 (7) \longrightarrow T_{\lambda(a)} \mfr{M}_3
\]
has maximal rank (i.e., 3) in a general point $a$. From the canonical isomorphism
\[
T_{\lambda(a)} \mfr{M}_3  = H^1 (Y, \, T_Y), \quad Y = \lambda (a), \quad T_Y = \text{tangent\ sheaf},
\]
given by the Kodaira-Spencer deformation class, we must calculate these deformation classes in a neighborhood of $a$.
Rather than working with $a$ directly, we work with the parameters $s, t, u, v$ that define our family of curves.
That is, we choose a convenient value of these parameters and consider first-order deformations around this value.

We calculate the deformation classes Hodge-theoretically. Namely in a family of curves $Y_{\alpha}$, depending on
a parameter $\alpha \in S$, we have a canonical connection
$$\nabla : H^1 _{DR} (Y_{\alpha}) \to  \Omega ^1 _{S/k}   \otimes H^1 _{DR} (Y_{\alpha})$$
which satisfies Griffiths transversality. In this case we get a map
\[
F^1 = H^0 (Y_{\alpha}, \Omega ^1 _{Y_{\alpha}})    \to F^0/F^1=  H^1 (Y_{\alpha}, \mathcal{O} _{Y_{\alpha}}) ,
\]
which is known to be cup-product with the Kodaira-Spencer class $\kappa (\alpha )\in H^1 (Y_{\alpha}, \, T_{Y_{\alpha}})$:
\[
 \Omega ^1 _{Y_{\alpha}}\otimes T_{Y_{\alpha}}\to \mathcal{O} _{Y_{\alpha}}:
H^0 (Y_{\alpha}, \Omega ^1 _{Y_{\alpha}}) \otimes H^1 (Y_{\alpha}, \, T_{Y_{\alpha}}) \to
H^1 (Y_{\alpha}, \mathcal{O} _{Y_{\alpha}}).
\]
To do this, we need representatives of the cohomology classes. We get these by considering the canonical models of our
genus 3 curves, which are projective quartics $(F_{\alpha}=0 )\subset \Proj ^2$. We have an isomorphism given by
Poincar\'e residue:
\[
H^2 _{DR} (\Proj ^2 - F_{\alpha}) \cong H^1 _{DR} (F_{\alpha}) .
\]
Griffiths showed that we can express the Hodge filtration of the (primitive) cohomology of a smooth projective hypersurface
in terms of the Jacobian ring of the hypersurface (see his papers, \cite{Gr}). In our case, this means that we have isomorphisms:

\[
R_1 \cong H^0 (Y_{\alpha}, \Omega ^1 _{Y_{\alpha}}) , R_5 \cong H^1 (Y_{\alpha}, \mathcal{O} _{Y_{\alpha}}) , \quad
R := k[x, y, z]/\left \langle \frac{\partial F_{\alpha}}{\partial x} ,    \frac{\partial F_{\alpha}}{\partial y} , \frac{\partial F_{\alpha}}{\partial z}   \right \rangle,
\]
with $x,y,z$ coordinates in projective space, and subscripts denoting the homogeneous components of that degree. Explicitly this means
that we can take
\[
\omega _1 = \frac {x \Omega}{F_{\alpha}}, \, \, \omega _2 = \frac {y \Omega}{F_{\alpha}}, \, \,  \omega _3 = \frac {z \Omega}{F_{\alpha}},
\quad \Omega = x dydz - ydxdz + zdxdy,
\]
as a basis of $F^1 = H^0 (Y_{\alpha}, \Omega ^1 _{Y_{\alpha}})  $, and

\[
\eta _1 = \frac {r_1 \Omega}{F_{\alpha}^2}, \, \, \eta _2 = \frac {r_2 \Omega}{F_{\alpha}^2}, \, \, \eta _3 = \frac {r_3 \Omega}{F_{\alpha}^2}, \quad
\text{for\  a basis \ } r_1, r_2, r_3 \text{\ of\ } R_5
\]
as a basis of $F^0/F^1 = H^1 (Y_{\alpha}, \mathcal{O} _{Y_{\alpha}})  $. We then compute the derivatives
$\partial \omega _i/\partial \alpha = g\Omega /F_{\alpha}^2$ and express these as linear combinations of the
$\omega _i,  \eta _j$ modulo exact forms. This is equivalent to computing the division of the polynomial $g$ with
respect to the ideal $\langle \frac{\partial F_{\alpha}}{\partial x} ,    \frac{\partial F_{\alpha}}{\partial y} , \frac{\partial F_{\alpha}}{\partial z}   \rangle$.
In fact, the remainder of division of $g$ by the Jacobian ideal will be in the shape $a_1 r_1 + a_2 r_2+a_3r_3$ and the deformation class is represented
by the differential form $a_1 \eta _1 + a_2\eta _2+a_3\eta _3$. Differentiating the classes
$\omega _1, \omega _2, \omega _3$ this way, for $\beta$ in a neighborhood of a fixed value $\alpha$, we get  a 3 by 3 matrix with entries in $k$.
Applied to our family, we chose $\alpha = (0, 1, 1, 0)$ and four deformations $\beta = (s, 1, 1, 0)$,   $\beta = (0, 1+t, 1, 0)$,
$\beta = (0, 1, 1+u, 0)$, $\beta = (0, 1, 1, v)$. Equations for these curves can be found in Appendix. 
For each of these, Magma calculated the canonical models of the curves in our family, giving a one-parameter family of smooth projective plane quartics
$F_{\beta}(x, y, z)$ to which we applied the above procedure. In this way we got 4 by 4 matrices with entries in $\QQ$, representing the
deformation classes in these 4 directions. These 4 matrices spanned a space of dimension 3, as expected. We have thus shown that the Jacobian has 
maximal rank in a specific point. It therefore has maximal rank in a Zariski open neighborhood of this point, and thus the image will contain a smoothly 
embedded locally closed subvariety of dimension 3 in $\mathfrak{M}_3$.


Another way to prove that this family has 3-dimensional moduli is the following: We can directly calculate the genus two curves
$y^2 = f_{\alpha}(x)$ and their dependence on parameters $\alpha = (s, t, u, v)$. From this we calculate the Igusa invariants
of the sextic polynomials $ f_{\alpha}(x)$. The function field of $ \mfr{M}_2$ is a rational function field on three independent
variables $j_1,j_2, j_3$. These functions in turn are explicit rational expressions in the Igusa invariants, and in this was we get
$j_1,j_2, j_3$ as rational functions of the parameters $s, t, u, v$. We compute the Jacobian matrix of $j_1,j_2, j_3$ with respect
to $s, t, u, v$ in a specific point  $s_0, t_0, u_0, v_0$ and find that it has rank 3. This is the rank of the tangent map of the
space of our parameters $s, t, u, v$ to the 3-dimensional moduli space $ \mfr{M}_2$. It has maximal rank at this one point, and therefore
in a Zariski neighborhood of this point. Thus the map from $(s, t, u, v)$-space  to  $ \mfr{M}_2$ has dense image. On the other hand, Ellenberg showed in \cite{JE} that the map from the moduli space of 6 distinct points in $\Proj ^1$ to the moduli space of genus 3
curves with $\QQ (\zeta _7 ^+)$-endomorphisms given by this construction also had maximal rank. The moduli space
of 6 distinct points in $\Proj ^1$ covers $ \mfr{M}_2$, since it is the moduli space of genus 2 curves with a level 2 structure. Therefore we again see that our construction leads to a family of genus 3 curves with 3 independent moduli.

\section{A polarization for our family}
\label{S:pol}
The genus 8 curves $X$ in our  family have a $D_7$-action. We will show how to construct a $D_7$-invariant polarization on the
Hodge structure $H^1(X)$. Recall that the representation of $D_7$ on the 16-dimensional space  $H^1(X, \QQ)$ is given by
$4\, \mr{alt}+2\alpha$ where $\alpha = \chi_1+\chi_2+\chi_3$. In fact the piece corresponding to $2\alpha$ is the motive
$ M= H^1(X)/H^1(C)$, where  $X\to C$ is the degree 7 unramified projection to the genus 2 curve $C$. It will suffice to put a $D_7$-invariant
principal polarization on $M$. As a $D_7$-module, $M = K^2$, where $K= \QQ (\zeta _7)$, where $s (x) = \zeta _7 x$,
$t (x) = \overline{x}$, complex conjugation.

\begin{proposition}
\label{P:10}
Let $K^2$ be as above. Consider the lattice $L = \ZZ_K \oplus P_7$, where $ \ZZ_K = \ZZ[\zeta_7]$ is the ring of integers
in $K$, and $P_7  =  w \ZZ_K, w = 1 - \zeta _7$,  is the unique prime ideal above $7$. Let $v= w \overline{w}$,
and $d^+$ the generator of the different ideal $\mathfrak {d}_{K^+/\QQ}$ of the real subfield
$K^+ = \QQ(\zeta _7+\overline{\zeta}_7)$, that is
\[
d^+ =  (\zeta _7+\overline{\zeta}_7)^2 + 3(\zeta _7+\overline{\zeta}_7) -3.
\]
The lattice $L$ is $D_7$-invariant.
Consider the bilinear form on $K^2$ given by the formula
\[
\langle  x   , y    \rangle  =
\langle  (x_1, x_2)   , (y_1, y_2 )    \rangle
 = \frac{1}{7} \mr{Tr}_{K/\QQ} \left (
 \frac{v^2(x_1 \overline{y}_2 -x_2 \overline{y}_1  )}{d^+}
 \right ).
\]
This bilinear form is alternating, $D_7$-invariant. Moreover, its restriction to the lattice $L$ is $\ZZ$-valued and with
elementary divisors all equal to $1$.
\end{proposition}

\begin{proof}
The lattice $L$ is a $D_7$ module: it is a sum of ideals, so stable under multiplication by $\zeta_7$; but is also invariant
by complex conjugation. This is clear for $\ZZ_K$, and since $P_7$ is the unique ideal above 7, $P_7$ is also stable
by complex conjugation.
Note that $c = v^2/d^+ \in K^+$, so it is invariant under complex conjugation. Complex conjugation belongs to the Galois group
of $K/\QQ$, so the trace map is invariant under complex conjugation. Then
\begin{align*}
7 \langle  x   , y    \rangle &= \mr{Tr}_{K/\QQ} \left ( c (x_1 \overline{y}_2 -x_2 \overline{y}_1  )\right )= \mr{Tr}_{K/\QQ} \left (\overline{ c (x_1 \overline{y}_2 -x_2 \overline{y}_1  )}\right )
=\mr{Tr}_{K/\QQ} \left ( c (y_2 \overline{x}_1 -y_1 \overline{x}_2  )\right )\\
&=-\mr{Tr}_{K/\QQ} \left ( c (y_1 \overline{x}_2 -y_2 \overline{x}_1  )\right ) = -7 \langle  x   , y    \rangle,
\end{align*}
so it is alternating. To see that it is $D_7$ invariant, we check the generators.
\[
7 \langle \zeta _7 x   , \zeta _7y    \rangle
= \mr{Tr}_{K/\QQ} \left ( c\zeta _7 \overline{\zeta}_7 (x_1 \overline{y}_2 -x_2 \overline{y}_1  )\right )
= \mr{Tr}_{K/\QQ} \left ( c (x_1 \overline{y}_2 -x_2 \overline{y}_1  )\right ) = 7 \langle x   , y    \rangle .
\]
Also
\begin{align*}
7 \langle  \overline{x}   , \overline{y}    \rangle &= \mr{Tr}_{K/\QQ} \left ( c (y_2 \overline{x}_1 -y_1 \overline{x}_2  )\right )= \mr{Tr}_{K/\QQ} \left (\overline{ c (y_2 \overline{x}_1 -y_1 \overline{x}_2  )}\right )=\mr{Tr}_{K/\QQ} \left ( c (x_1 \overline{y}_2 -x_2 \overline{y}_1  )\right )= 7 \langle  x   , y    \rangle.
\end{align*}
To see that $ \langle  x   , y    \rangle$ is $\ZZ$-valued on the lattice $L$, one computes the matrix of the bilinear
form in a $\ZZ$-basis of $L$. Magma then calculates the Frobenius normal form for this $\ZZ$-valued alternating form and finds a
standard alternating matrix with elementary divisors all equal to 1.
\end{proof}

\section{Finding polynomials for unramified degree $7$  covering  of genus $2$ curves}
\label{S:polys}
Consider the curve \[ h(x)^2-s(x)q(x)^2=(-g(x))^7, \quad
\text{where } \deg h=7, \, \deg s= 6, \, \deg q =4, \, \deg
g=2.\]To simplify our notation, denote the above curve as
\[ S_7^2 - F_6  \cdot Q_4^2 = R_2^7, \]
where  $S_7, F_6, Q_4, R_2$ are binary forms in $X$ and $Z$ with
the subscripts as the degrees,  and  the coefficients in an
algebraically closed field $\overline{K}$ with $\Char{K} \neq 2,
7$.

Our purpose is to find a dimension 3 moduli of unramified degree 7
covering
\[ W^7 = S_7 + Y \cdot Q_4, \quad \text{of genus 2 curves }
\, \,  Y^2 = F_6 .\]

\begin{remark}
 \label{R:2}
 To obtain the solution, we restrict our attention
to the curves that \item[1. ] $F_6$ must be square free in order
for the curve $Y^2=F_6$ to be of genus  $2$,

Moreover, $F_6$ does not have repeated roots, and  $\disc(F_6)
\neq 0$, the discriminant of $F_6$ is non-zero.

\item[2. ] $S_7 + Y \cdot Q_4$ must not be a power of $X,Y,Z$ in
order to have a degree $7$ covering.

\item[3. ] $\gcd(S_7^2, \,  F_6 \cdot Q_4^2) = 1$ or equivalently,
$\gcd( F_6 \cdot Q_4^2, \, R_2^7)=1$ in order for the covering to
be unramified.
\end{remark}

\begin{proof} To prove the last remark, we note that
for the covering to be unramified, the order of the common divisor
of ${S_7 + Y \cdot Q_4,  \, S_7 - Y \cdot Q_4} $ is either $0$ or
$7$ at any given point on the genus 2 curve.    This implies
 $\gcd(S_7^2, \,  F_6 \cdot Q_4^2)$ is either a constant, or  a $7$th power of a linear or a quadratic
factor. But since $F_6$ is square free, hence, it is impossible
for $\gcd(S_7^2, \,  F_6 \cdot Q_4^2)$ to be a $7$th power of a
quadratic factor.  Thus, $\gcd(S_7^2, \,  F_6 \cdot Q_4^2)$ is
either a constant, or  a $7$th power of a linear factor.

In the case of $7$th power of a linear factor, we can write
\[ S_7 = S_3 \cdot L_1^4 ,\quad F_6 = F_5 \cdot L_1 ,\quad Q_4 = Q_1
 \cdot L_1^3 ,\quad R_2 = R_1 \cdot  L_1, \quad \text{so that } \quad S_3^2 \cdot L_1 - F_5 \cdot Q_1^2 = R_1^7\]
where the  subindexes stand for the degrees. Perform a  linear
transformations, let \[ X=R_1 ,\quad Z =L_1,\quad  \text{ so that
} \quad  Q_1 = X - Z.\] Let $s_7(X) = S_7(X, 1) , \,  f_5(X) =
F_5(X, 1)$, then
\[ S_3^2 \cdot L_1 - F_5 \cdot
Q_1^2 = R_1^7  , \quad \Rightarrow \quad  s_3^2 (X) - f_5(X) \cdot
(X-1)^2 = X^7 .\]  And we have
\[s_3(1)^2 - f_5(1) \cdot (1-1)^2 = 1^7 \quad \Rightarrow \quad s_3 (1)^2 =
1 \quad \Rightarrow \quad s_3(1) = \pm1. \] Moreover, taking the
derivatives with respect to $X$, we have
\[ 2s_3(X) \cdot s'_3(X) - f'_5(X) \cdot (X-1)^2 -2f_5(X)(X-1)  =7 X^6
\quad \Rightarrow \quad  s_3 (1)\cdot s_3' (1) = 7, \quad
\Rightarrow \quad s_3'(1) = \pm 7.\] The  cubic polynomials $s_3$
form a 4 dimensional space, and the equations $s_3(1)=\pm 1$ and
$s_3'(1)=\pm 7$ determine a two dimensional subspace of $s_3$.
Since our goal is to find the dimension 3 moduli of unramified
degree 7 covering, we will not consider this case.

Therefore, we focus on the case $\gcd(S_7^2, \,  F_6 \cdot Q_4^2)
= 1$ or equivalently, $\gcd( F_6 \cdot Q_4^2, \, R_2^7)=1$.
\end{proof}

\begin{theorem}
\label{T:1}

\[ S_7(X,1)=\sum_{i=1}^4  \prod_{j=1, j \neq i}^4 \left( \dfrac{X-\beta^2_j}{\beta^2_i-\beta^2_j} \right)^2
\left( \beta_i^7 +  \left( \frac{7}{2}  \beta_i^5  - 2 \beta_i^7
l'_i(\beta_i^2) \right) \left(X-\beta_i^2 \right) \right), \text{
where } l_i(x) =\prod_{j=1, j \neq i}^4
\dfrac{x-\beta^2_j}{\beta^2_i-\beta^2_j},\] is a solution for  a
dimension $3$ moduli of unramified degree $7$ covering
\[ W^7=S_7 + Y \cdot Q_4, \quad \text{of genus $2$ curves }
\, \,  Y^2 = F_6 , \] where $\beta^2_i$ are distinct roots for
$Q_4$.
\end{theorem}

\begin{proof}
Apply a linear transformation so that $W=X\cdot Z$, and
\[
S_7^2 - F_6 \cdot Q_4^2 = W^7  \quad \Rightarrow \quad S_7^2 - F_6
\cdot Q_4^2 = (X\cdot Z)^7 .\] By Remark \ref{R:2}, 
$\gcd( F_6 \cdot Q_4^2, \, R_2^7)= \gcd( F_6 \cdot Q_4^2, \, W^7)=
\gcd( F_6 \cdot Q_4^2, \, (X \cdot Z)^7)=1$,  up to a constant
coefficient, $Q_4$ must be of the following form
\[ Q_4 = \prod_{i=1}^4 ( X - \alpha_i Z), \quad \alpha_i \neq 0,
\, \, i = 1, 2,3, 4.\] Let $s_7(X) = S_7(X, 1)$, then
\[ S_7^2 - F_6 \cdot Q_4^2 = R_2^7 , \quad \Rightarrow \quad
s_7^2 (X) - F_6(X,1)\prod_{i=1}^4 ( X - \alpha_i )^2 =X^7 .\]
Then, the polynomial $s_7$ satisfies the following system of
equations:
\[s_7^2(\alpha_i) = \alpha_i^7 ,\quad 2 \cdot s_7(\alpha_i)\cdot
s_7'(\alpha_i) = 7 \cdot\alpha_i^6 ,\quad
 i = 1, 2, 3, 4.\]
Choose a set of suitable $\beta_i$ such that $\beta_i^2= \alpha_i$
for $i = 1, 2, 3, 4$, then
\[ s_7(\beta_i^2) = \beta_i^7 ,\quad 2 \cdot s_7'(\beta_i) = 7 \cdot \beta_i^5
,\quad  i = 1, 2, 3, 4.\] This is a problem of Hermite
interpolation.
If we are to interpolate a given function $f(x)$ with a polynomial $p(x)$
so that
\[ p(x_i)=f(x_i), \, p'(x_i)=f'(x_i), \, 0 \leq i \leq n,\]
then the  Lagrange form of the Hermite interpolation polynomials
is given by
\begin{eqnarray*} p(x) &= & \sum_{i=0}^n f(x_i) l_i^2(x)\left(  1-2 l'_i(x_i) (x-x_i) \right)
+ \sum_{i=0}^n f'(x_i) l_i^2(x)(x-x_i), \\
& = &  \sum_{i=0}^n  l_i^2(x) \left( f(x_i)   + \left(  f'(x_i) -2
f(x_i) l'_i(x_i) \right) \left(x-x_i \right) \right),  \,  \,
\text{ where } \, \, l_i(x) =\prod_{j=0, j \neq i}^n
\dfrac{x-x_j}{x_i-x_j}.\end{eqnarray*}

Apply this result in our situation, we obtain for $i=1, 2, 3, 4$,
\[ x_i=\beta^2_i, \, \, p(x_i) = s_7(\beta_i^2), \, \,  f(x_i)=  \beta_i^7,  \, \, f'(x_i) = \frac{7}{2}  \beta_i^5,  \, \,
l_i(x) =\prod_{j=1, j \neq i}^4
\dfrac{x-\beta^2_j}{\beta^2_i-\beta^2_j}\]
\[ S_7(X,1)=s_7(X)=\sum_{i=1}^4  \prod_{j=1, j \neq i}^4 \left( \dfrac{X-\beta^2_j}{\beta^2_i-\beta^2_j} \right)^2
\left( \beta_i^7 +  \left( \frac{7}{2}  \beta_i^5  - 2 \beta_i^7
l'_i(\beta_i^2) \right) \left(X-\beta_i^2 \right) \right) . \]
\end{proof}

\begin{theorem}
\label{T:2}

With above notation,
the degree $7$ unramified cover of the curve
$Y^2 = F_6(X,Z)$
has the following defining equation:
\[    T^7 - 7 T^5 X  Z+ 14 T^3 X^2 Z^2 - 7 T X^3  Z^3 -2 S_7(X,Z).\]
In the affine coordinates, \[    K(X, Y, T) = K(X, Y, W), \quad \text{and}
\quad    [K(X, T): K(X)] = 7. \]
\end{theorem}
\begin{proof}
Consider the equation $W^7 = S_7(X, Z) + Y \cdot Q_4(X,Z)$
where $Y^2 = F_6(X,Z).$  This is the degree $7$ unramified cover of the curve
$    Y^2 = F_6(X, Z).$
Eliminate $Y$, we get
\[    (W^7 - S_7(X, Z))^2 - F_6(X,Z) \cdot Q_4(X,Z)^2 = 0,\] which is equivalent to
\begin{equation*}\label{genus8}
    W^{14} - 2S_7(X, Z) W^7 +X^7 Z^7 = 0, \quad \text{where the coefficients are rational functions of
$\beta_i$ ($i=1, 2,..., 4$).}
\end{equation*}
Consider the polynomial
\[ J =  \prod_{i=0}^6 (T - \zeta_7^i \cdot W - \zeta_7^{- i} \cdot \bar{W}) , \quad \text{ where
$W \bar{W} = X.$} \]
With Mathematica, we calculate:
\[    J = T^7 - 7 T^5 X  Z+ 14 T^3 X^2 Z^2 - 7 T X^3  Z^3 -2 S_7(X,Z).\] Hence, the degree $7$ unramified cover of the curve $   Y^2 = F_6(X,Z)$
has the following defining equation:
\[    T^7 - 7 T^5 X  Z+ 14 T^3 X^2 Z^2 - 7 T X^3  Z^3 -2 S_7(X,Z). \]
We claim that in the affine coordinates, we have
\[    K(X, Y, T) = K(X, Y, W), \quad \text{ and } \quad     [K(X, T): K(X)] = 7. \]
First, let $K_1$ be a field of characteristic other than $7$ and  $K_2$
its quadratic extension. Let $\omega$ be a non-zero element of $K_2$ such that its norm
$\omega\bar\omega$ to $K_1$ equals to a $7$th power $\rho^7$ with
$\rho\in K_1$.  Let \[ \Omega= \sqrt[7]{\omega} \in \overline{K}_2, \quad \bar\Omega = \rho/\Omega \quad \text{so that} \quad \bar\Omega^7 = \rho^7/\Omega^7=\rho^7/\omega =\bar\omega. \]

By Kummer theory, the extension $K_2(\zeta, \Omega) / K_2(\zeta)$, is a cyclic extension of degree $7$.
Let $T=\Omega+\bar\Omega$, then
\[    K_2(\zeta_7) \subset K_2(\zeta_7, T) \subset K_2(\zeta_7, \Omega).\]
Since $\Omega$ is a linear
combination of $T = \Omega + \bar\Omega$ and its conjugate
$\zeta_7\Omega+ \zeta_7^{-1}\bar\Omega$, where coefficients belong
to $K_2(\zeta_7)$.  Thus, we indeed have
\[    K_2(\zeta_7, T) = K_2(\zeta_7, \Omega).\]
Also, on one hand \[    [K_2(\zeta_7, T) : K_2] = [K_2(\zeta_7, \Omega): K_2] = [K_2(\zeta_7, \Omega): K_2(\zeta_7)] [K_2(\zeta_7):K_2] =7[K_2(\zeta_7):K_2],\]
and on the other hand,\[    [K_2(\zeta_7, T) : K_2] = [K_2(\zeta_7, T) : K_2(T)]\cdot [K_2(T): K_2].\]
The index $[K_2(\zeta_7, T) : K_2(T)]$ is a divisor of $6$ implies that $7 \mid  [K_2(T): K_2].$   The fact that the polynomial \[ \prod_{i=0}^6 (t - \zeta_7^i \cdot \Omega - \zeta_7^{- i} \cdot \bar{\Omega}), \, \text{ has expansion } \, t^7 - 7 t^5 \rho  + 14 t^3 \rho^2 - 7 t \rho^3 -2 (\omega+ \bar\omega) \, \,  \text{ with coefficients in $K_2$},\] yields
\[  [K_2(T): K_2]= 7.\]
As $[K_1(T):K_1] \ge [K_2(T):K_2]$, this further implies
\[    [K_1(T):K_1] = 7.\]
Hence, we proved our claim $
[K(X, T): K(X)] = 7$.

Since $K_2(T)\subset K_2(\Omega)$,
the equality $[K_2(T): K_2]= 7$ implies \[ K(X, Y, T) = K(X, Y, W). \]

\end{proof}

Follow the material in \cite[p245-246 Theorem 5.3.5]{HenriCohen}, let $K_1 =
\mathbb{Q}(X)$, and restrict the values of $\beta_i$'s to obtain a cyclic extensions of
degree $7$ of the field $K_2=\mathbb{Q}(X, Y)$.

First, we note any extension $K_2(\sqrt[7]{\xi})/K_2$
 can not be Galois when $\xi \in K_2^{\times}$.  Since $K_2 \cap \mathbb{Q}(\zeta_7) = \mathbb{Q}$, the extension $K_2(\zeta_7) / K_2$ is a cyclic extension of degree $6$.

Let $\xi \in K_2(\zeta_7)^{\times}$
and consider its $7$th root $\Xi$ such that $\Xi^7 = \xi. $
To study  the orbit of $\Xi$ under the Galois group of
$K_2(\zeta_7, \Xi) / K_2$, we let a conjugation $\sigma$ of
$K_2(\zeta_7, \Xi) / K_2$ such that
\[\zeta_7^{\sigma}  = \zeta_7^3. \]
We also let the conjugation $\tau$ of $K_2(\zeta_7, \Xi) /
K_2$ such that
\[ \zeta_7^{\tau} = \zeta_7,\quad \Xi^{\tau}=\zeta_7\Xi. \]
Then,\[    \Xi^{\tau^j \sigma^i} = (\zeta_7^j \Xi)^{\sigma^i} = \zeta_7^{3^ij} \Xi^{\sigma^i} .\]

Assume the extension $K_2(\zeta_7, \Xi)$ contains a cyclic
extension of degree $7$ of $K_2$. Then, the extension
$K_2(\zeta_7, \Xi)/K_2$ is abelian. Therefore, we must have
\[    \Xi^{\sigma^i \tau^j} = \Xi^{\tau^j \sigma^i} =  \zeta_7^{3^ij} \Xi^{\sigma^i} .\]
In particular,
\[    (\Xi^{2\sigma} \Xi)^{\tau} = \Xi^{2\sigma} \Xi.\]
This means
\[    \Xi^{2\sigma} \Xi \in K_2(\zeta_7), \, \text{ or equivalently } \,
    \xi^{2\sigma} \xi \in (K_2(\zeta_7)^{\times})^7.
\end{equation*}

On the other hand, every subextension of $K_2(\zeta_7, \Xi)/K_2$
is abelian, which suggests $
    K_2(\xi) = K_2(\zeta_7).$  Otherwise, $K_2(\Xi)/K_2(\xi)$ is non-Galois would
imply $K_2(\Xi)/K_2$ is non-abelian, contradicting our assumption.

Therefore,
\begin{equation*}\label{Non-degeneracy}
    \xi \notin (K_2(\zeta_7)^\times)^7.
\end{equation*}

Consider the sum
\[    \theta_j = \sum_{i=0}^5 \Xi^{\sigma^i \tau^j}
    =
    \sum_{i=0}^5 \zeta_7^{3^i j} \Xi^{\sigma^i}
    ,\quad (j=0, 1, ..., 6)
    .\]
Each $\theta_j$ is invariant under $\sigma$. Hence, we have either
$[K_2(\theta_0)/K_2] = 7$ or $1$. The latter possibility is
eliminated by the following matrix:
\[    (\zeta_7^{3^i j} - 1)_{0 \le i \le 5,\;\; 1 \le j \le 6} .\]



\section{Appendix: explicit equations}
\label{S:explicit}

Recall from section \ref{S:constr} 
we found equations for our curves that depended on parameters,
$u_1, u_2, u_3, u_4$. In fact, the equations are symmetric functions of these, so we can express
them in terms of the elementary symmetric functions. We use Mathematica to solve the equation system. Let
\[ \alpha = \sum_{i=1}^4 u_i,  \, \,  \beta=\sum_{i\neq j} u_iu_j,  \, \,  \gamma = \sum_{i \neq j \neq k} u_iu_ju_k,  \, \,  \delta = u_1u_2u_3u_4.\]

The general equation of the genus 3 curves is
$$ Y(u_1, u_2, u_3, u_4): z^7 - 7xz^5 + 14x^2z^3 - 7x^3z - 2h(x)=0$$
where

$h(x)= (\alpha^3\gamma^2\delta^4+\gamma^3\delta^4+\alpha^3\beta\delta^5-3\alpha^2\gamma\delta^5+
(-2\alpha^3\gamma^4\delta^2-2\gamma^5\delta^2+2\alpha^3\beta\gamma^2\delta^3+6\alpha^2\gamma^3\delta^3+
4\beta\gamma^3\delta^3+3\alpha^3\beta^2\delta^4-\alpha^4\gamma\delta^4-9\alpha^2\beta\gamma\delta^4-
3\alpha\gamma^2\delta^4+\alpha^3\delta^5)x+(\alpha^3\gamma^6+\gamma^7-3\alpha^3\beta\gamma^4\delta-
3\alpha^2\gamma^5\delta-4\beta\gamma^5\delta+4\alpha^3\beta^2\gamma^2\delta^2-2\alpha^4\gamma^3\delta^2+
6\alpha^2\beta\gamma^3\delta^2+6\beta^2\gamma^3\delta^2+2\alpha\gamma^4\delta^2+3\alpha^3\beta^3\delta^3-
8\alpha^4\beta\gamma\delta^3-9\alpha^2\beta^2\gamma\delta^3+14\alpha^3\gamma^2\delta^3-9\alpha\beta\gamma^2\delta^3+
3\gamma^3\delta^3+\alpha^5\delta^4+6\alpha^3\beta\delta^4-9\alpha^2\gamma\delta^4)x^2+(-3\alpha^3\beta^2\gamma^4+
3\alpha^4\gamma^5+3\alpha^2\beta\gamma^5-2\beta^2\gamma^5+\alpha\gamma^6+3\alpha^3\beta^3\gamma^2\delta+
3\alpha^2\beta^2\gamma^3\delta+4\beta^3\gamma^3\delta-15\alpha^3\gamma^4\delta-5\alpha\beta\gamma^4\delta-
\gamma^5\delta+\alpha^3\beta^4\delta^2-7\alpha^4\beta^2\gamma\delta^2-3\alpha^2\beta^3\gamma\delta^2+
14\alpha^3\beta\gamma^2\delta^2-9\alpha\beta^2\gamma^2\delta^2+19\alpha^2\gamma^3\delta^2+9\beta\gamma^3\delta^2-
\alpha^5\beta\delta^3+9\alpha^3\beta^2\delta^3+\alpha^4\gamma\delta^3-18\alpha^2\beta\gamma\delta^3-
9\alpha\gamma^2\delta^3+3\alpha^3\delta^4)x^3+(-3\alpha^4\beta^2\gamma^3+3\alpha^2\beta^3\gamma^3+\beta^4\gamma^3+
3\alpha^5\gamma^4-7\alpha\beta^2\gamma^4-\beta\gamma^5+3\alpha^5\beta\gamma^2\delta+3\alpha^3\beta^2\gamma^2\delta-
3\alpha\beta^3\gamma^2\delta-15\alpha^4\gamma^3\delta+14\alpha^2\beta\gamma^3\delta+9\beta^2\gamma^3\delta+
\alpha\gamma^4\delta-2\alpha^5\beta^2\delta^2+4\alpha^3\beta^3\delta^2+\alpha^6\gamma\delta^2-
5\alpha^4\beta\gamma\delta^2-9\alpha^2\beta^2\gamma\delta^2+19\alpha^3\gamma^2\delta^2-18\alpha\beta\gamma^2\delta^2+
3\gamma^3\delta^2-\alpha^5\delta^3+9\alpha^3\beta\delta^3-9\alpha^2\gamma\delta^3)x^4+(\alpha^6\gamma^3-
3\alpha^4\beta\gamma^3+4\alpha^2\beta^2\gamma^3+3\beta^3\gamma^3-2\alpha^3\gamma^4-8\alpha\beta\gamma^4+\gamma^5-
3\alpha^5\gamma^2\delta+6\alpha^3\beta\gamma^2\delta-9\alpha\beta^2\gamma^2\delta+14\alpha^2\gamma^3\delta+
6\beta\gamma^3\delta+\alpha^7\delta^2-4\alpha^5\beta\delta^2+6\alpha^3\beta^2\delta^2+2\alpha^4\gamma\delta^2-
9\alpha^2\beta\gamma\delta^2-9\alpha\gamma^2\delta^2+3\alpha^3\delta^3)x^5+(-2\alpha^4\gamma^3+2\alpha^2\beta\gamma^3+
3\beta^2\gamma^3-\alpha\gamma^4+6\alpha^3\gamma^2\delta-9\alpha\beta\gamma^2\delta+\gamma^3\delta-2\alpha^5\delta^2+
4\alpha^3\beta\delta^2-3\alpha^2\gamma\delta^2)x^6+(\alpha^2\gamma^3+\beta\gamma^3-3\alpha\gamma^2\delta+
\alpha^3\delta^2)x^7)/(2(-\alpha\beta\gamma+\gamma^2+\alpha^2\delta)^3).$

\quad

The equation genus 2 curve is
\begin{equation*}
S_6(x) = \sum_{i=0}^6 a_ix^i
\end{equation*}
with

$a_0= \beta ^2 \delta ^6 \alpha ^6+2 \beta  \gamma ^2 \delta ^5 \alpha ^6+\gamma ^4 \delta ^4 \alpha ^6-6 \beta  \gamma  \delta ^6 \alpha ^5-6 \gamma ^3 \delta ^5 \alpha ^5+9 \gamma
   ^2 \delta ^6 \alpha ^4+2 \beta  \gamma ^3 \delta ^5 \alpha ^3+2 \gamma ^5 \delta ^4 \alpha ^3-6 \gamma ^4 \delta ^5 \alpha ^2+\gamma ^6 \delta ^4, $

$\,$

$a_1= -2 \delta ^2 \gamma ^8-4 \alpha ^3 \delta ^2 \gamma ^7+12 \alpha ^2 \delta ^3 \gamma ^6+4 \beta  \delta ^3 \gamma ^6-2 \alpha ^6 \delta ^2 \gamma ^6-6 \alpha  \delta ^4 \gamma
   ^5+12 \alpha ^5 \delta ^3 \gamma ^5+4 \alpha ^3 \beta  \delta ^3 \gamma ^5-26 \alpha ^4 \delta ^4 \gamma ^4-18 \alpha ^2 \beta  \delta ^4 \gamma ^4+20 \alpha ^3 \delta ^5
   \gamma ^3-2 \alpha ^7 \delta ^4 \gamma ^3+6 \alpha ^3 \beta ^2 \delta ^4 \gamma ^3-6 \alpha ^5 \beta  \delta ^4 \gamma ^3+8 \alpha ^6 \delta ^5 \gamma ^2+12 \alpha ^4 \beta
   \delta ^5 \gamma ^2+4 \alpha ^6 \beta ^2 \delta ^4 \gamma ^2-6 \alpha ^5 \delta ^6 \gamma -12 \alpha ^5 \beta ^2 \delta ^5 \gamma -2 \alpha ^7 \beta  \delta ^5 \gamma +2
   \alpha ^6 \beta  \delta ^6+2 \alpha ^6 \beta ^3 \delta ^5 ,$

$\,$

$a_2=
 \gamma ^{10}+2 \alpha ^3 \gamma ^9+\alpha ^6 \gamma ^8-6 \alpha ^2 \delta  \gamma ^8-4 \beta  \delta  \gamma ^8+8 \alpha  \delta ^2 \gamma ^7-6 \alpha ^5 \delta  \gamma ^7-6
   \alpha ^3 \beta  \delta  \gamma ^7+2 \delta ^3 \gamma ^6+17 \alpha ^4 \delta ^2 \gamma ^6+6 \beta ^2 \delta ^2 \gamma ^6+12 \alpha ^2 \beta  \delta ^2 \gamma ^6-2 \alpha ^6
   \beta  \delta  \gamma ^6-10 \alpha ^3 \delta ^3 \gamma ^5-18 \alpha  \beta  \delta ^3 \gamma ^5+8 \alpha ^3 \beta ^2 \delta ^2 \gamma ^5+6 \alpha ^5 \beta  \delta ^2 \gamma
   ^5-3 \alpha ^2 \delta ^4 \gamma ^4+12 \alpha ^6 \delta ^3 \gamma ^4-18 \alpha ^2 \beta ^2 \delta ^3 \gamma ^4-22 \alpha ^4 \beta  \delta ^3 \gamma ^4+3 \alpha ^6 \beta ^2
   \delta ^2 \gamma ^4-34 \alpha ^5 \delta ^4 \gamma ^3+46 \alpha ^3 \beta  \delta ^4 \gamma ^3+6 \alpha ^3 \beta ^3 \delta ^3 \gamma ^3-18 \alpha ^5 \beta ^2 \delta ^3 \gamma
   ^3-12 \alpha ^7 \beta  \delta ^3 \gamma ^3+12 \alpha ^4 \delta ^5 \gamma ^2+3 \alpha ^8 \delta ^4 \gamma ^2-3 \alpha ^4 \beta ^2 \delta ^4 \gamma ^2+50 \alpha ^6 \beta
   \delta ^4 \gamma ^2+6 \alpha ^6 \beta ^3 \delta ^3 \gamma ^2-8 \alpha ^7 \delta ^5 \gamma -24 \alpha ^5 \beta  \delta ^5 \gamma -6 \alpha ^5 \beta ^3 \delta ^4 \gamma -10
   \alpha ^7 \beta ^2 \delta ^4 \gamma +\alpha ^6 \delta ^6+6 \alpha ^6 \beta ^2 \delta ^5+2 \alpha ^8 \beta  \delta ^5+\alpha ^6 \beta ^4 \delta ^4 ,$

$\,$

$a_3=
 -2 \gamma  \delta ^4 \alpha ^9+2 \delta ^5 \alpha ^8-2 \beta ^2 \delta ^4 \alpha ^8+6 \beta  \gamma ^2 \delta ^3 \alpha ^8+2 \gamma ^7 \alpha ^7-2 \beta  \gamma  \delta ^4
   \alpha ^7-6 \gamma ^3 \delta ^3 \alpha ^7-6 \beta ^2 \gamma ^3 \delta ^2 \alpha ^7+6 \beta  \gamma ^5 \delta  \alpha ^7-4 \beta ^2 \gamma ^6 \alpha ^6+6 \beta  \delta ^5
   \alpha ^6+4 \beta ^3 \delta ^4 \alpha ^6+4 \gamma ^2 \delta ^4 \alpha ^6+6 \beta ^2 \gamma ^2 \delta ^3 \alpha ^6-18 \beta  \gamma ^4 \delta ^2 \alpha ^6-20 \gamma ^6 \delta
    \alpha ^6+6 \beta  \gamma ^7 \alpha ^5-12 \gamma  \delta ^5 \alpha ^5-18 \beta ^2 \gamma  \delta ^4 \alpha ^5+46 \beta  \gamma ^3 \delta ^3 \alpha ^5+60 \gamma ^5 \delta ^2
   \alpha ^5+6 \beta ^3 \gamma ^3 \delta ^2 \alpha ^5+18 \beta ^2 \gamma ^5 \delta  \alpha ^5-6 \beta  \gamma ^2 \delta ^4 \alpha ^4-92 \gamma ^4 \delta ^3 \alpha ^4-6 \beta ^3
   \gamma ^2 \delta ^3 \alpha ^4-56 \beta ^2 \gamma ^4 \delta ^2 \alpha ^4-18 \beta  \gamma ^6 \delta  \alpha ^4-6 \beta ^2 \gamma ^7 \alpha ^3+40 \gamma ^3 \delta ^4 \alpha
   ^3+32 \beta ^2 \gamma ^3 \delta ^3 \alpha ^3+46 \beta  \gamma ^5 \delta ^2 \alpha ^3+2 \beta ^4 \gamma ^3 \delta ^2 \alpha ^3-6 \gamma ^7 \delta  \alpha ^3+6 \beta ^3 \gamma
   ^5 \delta  \alpha ^3+6 \beta  \gamma ^8 \alpha ^2-6 \beta  \gamma ^4 \delta ^3 \alpha ^2+4 \gamma ^6 \delta ^2 \alpha ^2-6 \beta ^3 \gamma ^4 \delta ^2 \alpha ^2+6 \beta ^2
   \gamma ^6 \delta  \alpha ^2-2 \gamma ^9 \alpha -12 \gamma ^5 \delta ^3 \alpha -18 \beta ^2 \gamma ^5 \delta ^2 \alpha -2 \beta  \gamma ^7 \delta  \alpha -2 \beta ^2 \gamma
   ^8+6 \beta  \gamma ^6 \delta ^2+2 \gamma ^8 \delta +4 \beta ^3 \gamma ^6 \delta , $

$\,$

$a_4=
 \delta ^4 \alpha ^{10}+2 \gamma ^3 \delta ^2 \alpha ^9+\gamma ^6 \alpha ^8-4 \beta  \delta ^4 \alpha ^8-6 \gamma ^2 \delta ^3 \alpha ^8+8 \gamma  \delta ^4 \alpha ^7-6 \beta
   \gamma ^3 \delta ^2 \alpha ^7-6 \gamma ^5 \delta  \alpha ^7-2 \beta  \gamma ^6 \alpha ^6+2 \delta ^5 \alpha ^6+6 \beta ^2 \delta ^4 \alpha ^6+12 \beta  \gamma ^2 \delta ^3
   \alpha ^6+17 \gamma ^4 \delta ^2 \alpha ^6-18 \beta  \gamma  \delta ^4 \alpha ^5-10 \gamma ^3 \delta ^3 \alpha ^5+8 \beta ^2 \gamma ^3 \delta ^2 \alpha ^5+6 \beta  \gamma ^5
   \delta  \alpha ^5+3 \beta ^2 \gamma ^6 \alpha ^4-3 \gamma ^2 \delta ^4 \alpha ^4-18 \beta ^2 \gamma ^2 \delta ^3 \alpha ^4-22 \beta  \gamma ^4 \delta ^2 \alpha ^4+12 \gamma
   ^6 \delta  \alpha ^4-12 \beta  \gamma ^7 \alpha ^3+46 \beta  \gamma ^3 \delta ^3 \alpha ^3-34 \gamma ^5 \delta ^2 \alpha ^3+6 \beta ^3 \gamma ^3 \delta ^2 \alpha ^3-18 \beta
   ^2 \gamma ^5 \delta  \alpha ^3+3 \gamma ^8 \alpha ^2+6 \beta ^3 \gamma ^6 \alpha ^2+12 \gamma ^4 \delta ^3 \alpha ^2-3 \beta ^2 \gamma ^4 \delta ^2 \alpha ^2+50 \beta
   \gamma ^6 \delta  \alpha ^2-10 \beta ^2 \gamma ^7 \alpha -24 \beta  \gamma ^5 \delta ^2 \alpha -8 \gamma ^7 \delta  \alpha -6 \beta ^3 \gamma ^5 \delta  \alpha +2 \beta
   \gamma ^8+\beta ^4 \gamma ^6+\gamma ^6 \delta ^2+6 \beta ^2 \gamma ^6 \delta , $

$\,$

$a_5= -2 \delta ^4 \alpha ^8-4 \gamma ^3 \delta ^2 \alpha ^7-2 \gamma ^6 \alpha ^6+4 \beta  \delta ^4 \alpha ^6+12 \gamma ^2 \delta ^3 \alpha ^6-6 \gamma  \delta ^4 \alpha ^5+4
   \beta  \gamma ^3 \delta ^2 \alpha ^5+12 \gamma ^5 \delta  \alpha ^5-18 \beta  \gamma ^2 \delta ^3 \alpha ^4-26 \gamma ^4 \delta ^2 \alpha ^4-2 \gamma ^7 \alpha ^3+20 \gamma
   ^3 \delta ^3 \alpha ^3+6 \beta ^2 \gamma ^3 \delta ^2 \alpha ^3-6 \beta  \gamma ^5 \delta  \alpha ^3+4 \beta ^2 \gamma ^6 \alpha ^2+12 \beta  \gamma ^4 \delta ^2 \alpha ^2+8
   \gamma ^6 \delta  \alpha ^2-2 \beta  \gamma ^7 \alpha -6 \gamma ^5 \delta ^2 \alpha -12 \beta ^2 \gamma ^5 \delta  \alpha +2 \beta ^3 \gamma ^6+2 \beta  \gamma ^6 \delta ,$

$\,$

$a_6=\delta ^4 \alpha ^6+2 \gamma ^3 \delta ^2 \alpha ^5+\gamma ^6 \alpha ^4-6 \gamma ^2 \delta ^3 \alpha ^4+2 \beta  \gamma ^3 \delta ^2 \alpha ^3-6 \gamma ^5 \delta  \alpha ^3+2
   \beta  \gamma ^6 \alpha ^2+9 \gamma ^4 \delta ^2 \alpha ^2-6 \beta  \gamma ^5 \delta  \alpha +\beta ^2 \gamma ^6 .$


It turns out that it is hard to use a computer to get the canonical forms of our family of curves. But we can get the following families of curves with their canonical forms, that is, the quartic equations of the genus 3 curves.

Notation as the previous subsection. Consider the curve
\[Y(0,1,1,0)=z^7 - 7xz^5 + 14x^2z^3 - 7x^3z - 2\left(\frac{x^7}{2}+\frac{3 x^6}{2}+2 x^5-x^3+\frac{x^2}{2}\right)=0.\]
Using Magma, we can compute its canonical form as a quartic in the projective plane ${\mathbb P}^2_K$. Assume the coordinates in ${\mathbb P}^2_K$ are $X, Y$ and $Z$, we have the equation
\begin{equation*}\begin{split}
X^4 + 8X^3Z + 2X^2YZ + 25X^2Z^2 - XY^3 + 2XY^2Z + 8XYZ^2 +
36XZ^3 + Y^4 - 2Y^3Z  \\ + 5Y^2Z^2 + 9YZ^3 + 20Z^4 =0 \end{split}
\end{equation*}
Consider the family
$$Y(s,1,1,0)=z^7 - 7xz^5 + 14x^2z^3 - 7x^3z - 2h(s , 1, 1, 0)$$
where
\begin{equation*}\begin{split}  h(s , 1, 1, 0)= & -\frac{1}{2 (s-1)^3}\left(s^3+1\right) x^2+\left(s^2+1\right) x^7-\left(2 s^4-2
   s^2+s-3\right) x^6 + \left(3 s^5-3 s^4+3 s^2-7 s\right) x^4+ \\ &  \left(3
   s^4-3 s^3+3 s^2+s-2\right) x^3+\left(s^6-3 s^4-2 s^3+4 s^2-8
   s+4\right) x^5) .\end{split}\end{equation*}

The canonical model of this family in the projective plane ${\mathbb P}^2_K$ is

\begin{equation*}\begin{split}  &  S(s,X,Y,Z) = \frac{\left(2 s-2 s^2\right) X^3 Y}{s^3+1}+ \frac{\left(-s^3+2
   s^2-s\right) X^2 Y^2}{s^5+s^3+s^2+1}+ \frac{\left(6 s^4-6 s^3-10
   s^2+8 s+2\right) X^2 Y Z}{s^3+1}+\\ & \frac{\left(2 s^5-4 s^4-4 s^3+14
   s^2-10 s+2\right) X Y^2 Z}{s^5+s^3+s^2+1}+\frac{\left(-4 s^7+4
   s^5-4 s^4+7 s^3+6 s^2-s+8\right) X^3
   Z}{s^5+s^3+s^2+1}+\\ & \frac{\left(-6 s^6+6 s^5+20 s^4-16 s^3-20 s^2+8
   s+8\right) X Y Z^2}{s^3+1}+ \frac{\left(2 s^5-6 s^4+7 s^3-5 s^2+3
   s-1\right) X Y^3}{s^8+s^6+2 s^5+2 s^3+s^2+1}+ \\ & \frac{\left(s^4-4
   s^3+6 s^2-4 s+1\right) Y^4}{s^8+s^6+2 s^5+2
   s^3+s^2+1}+\frac{\left(6 s^9-18 s^7+6 s^6+3 s^5-24 s^4+20 s^3+15
   s^2-9 s+25\right) X^2 Z^2}{s^5+s^3+s^2+1}+\\ & \frac{\left(-2 s^7+6
   s^6-5 s^5+s^4-s^3-s^2+4 s-2\right) Y^3 Z}{s^8+s^6+2 s^5+2
   s^3+s^2+1}+\\ & \frac{\left(-4 s^{11}+20 s^9-4 s^8-27 s^7+26 s^6-5
   s^5-54 s^4+34 s^3+18 s^2-24 s+36\right) X
   Z^3}{s^5+s^3+s^2+1}+\\ & \frac{\left(-s^{10}+2 s^9+5 s^8-15 s^7+5
   s^6+19 s^5-23 s^4-2 s^3+25 s^2-20 s+5\right) Y^2 Z^2}{s^8+s^6+2
   s^5+2 s^3+s^2+1}+\\ & \frac{\left(s^{13}-7 s^{11}+s^{10}+17 s^9-9
   s^8-14 s^7+31 s^6-9 s^5-43 s^4+28 s^3+8 s^2-20 s+20\right)
   Z^4}{s^5+s^3+s^2+1}+\\ & \frac{\left(2 s^{13}-2 s^{12}-8 s^{11}+8
   s^{10}+6 s^9-8 s^8+8 s^7-9 s^5+7 s^4-10 s^3+2 s^2-5 s+9\right) Y
   Z^3}{s^8+s^6+2 s^5+2 s^3+s^2+1}+X^4=0.\end{split}\end{equation*}
We also have another three families of canonical models. For
$$Y(0,1+t,1,0) =z^7 - 7xz^5 + 14x^2z^3 - 7x^3z - 2h(0 , 1+t, 1, 0)$$ where
\begin{eqnarray*} & & h(0 , 1+t, 1, 0) \\
& = & \frac{1}{2} (t+1) x^7+\frac{3}{2} (t+1)^2 x^6+\frac{1}{2} \left(3
   (t+1)^3+1\right) x^5+  \frac{1}{2} \left((t+1)^4-t-1\right)
   x^4-(t+1)^2 x^3+\frac{x^2}{2},\end{eqnarray*}
the canonical form of the family $Y(0,1+t,1,0)$ is
\begin{equation*}\begin{split}   &T(t,X,Y,Z)= \frac{\left(24 t^3+72 t^2+72 t+25\right) X^2 Z^2}{t+1}+\left(32
   t^3+96 t^2+96 t+36\right) X Z^3+ \\ &\frac{\left(8 t^3+24 t^2+24
   t+9\right) Y Z^3}{t+1}+\left(16 t^4+64 t^3+96 t^2+68 t+20\right)
   Z^4+(8 t+8) X^3 Z+ \\& \frac{2 X Y^2 Z}{t+1}+(8 t+8) X Y
   Z^2+\frac{Y^4}{t+1}+(-2 t-2) Y^3 Z+X^4+2 X^2 Y Z-X Y^3+5 Y^2 Z^2=0.\end{split}\end{equation*}
For
$$Y(0,1,1+u,0)=z^7 - 7xz^5 + 14x^2z^3 - 7x^3z - 2h(0 , 1, 1+u, 0)$$ where
\begin{equation*}\begin{split}  h(0 , 1, 1+u, 0)  = & \frac{1}{2 (u+1)^6}((u+1)^3 x^7+3 (u+1)^3 x^6+\left((u+1)^5+3 (u+1)^3\right) x^5- \\ &\left((u+1)^5-(u+1)^3\right) x^4-2 (u+1)^5 x^3+(u+1)^7 x^2),\end{split}\end{equation*}
the canonical form of the family $Y(0,1,1+u,0)$ is
\begin{equation*}\begin{split}   &U(u,X,Y,Z)= \left(u^2+2 u+25\right) X^2 Z^2+\left(2 u^2+4 u+2\right) X Y^2 Z+\left(4 u^2+8 u+36\right) X Z^3+ \\& \left(u^2+2 u+1\right) Y^4+\left(5 u^2+10 u+5\right) Y^2 Z^2+\left(4 u^2+8 u+20\right) Z^4+\left(u^3+3 u^2+11 u+9\right) Y Z^3+ \\ &(2 u+2) X^2 Y Z+(-u-1) X Y^3+(8 u+8) X Y Z^2+(-2 u-2) Y^3 Z+X^4+8 X^3 Z=0.\end{split}\end{equation*}

For
$$Y(0,1,1,v)=z^7 - 7xz^5 + 14x^2z^3 - 7x^3z - 2h(0 , 1, 1, v)$$ where
\begin{equation*}\begin{split}  h(0 , 1, 1, v)= & \frac{v^4}{2}+\frac{3}{2} \left(v^2+3 v\right) x^4+\frac{1}{2}
   \left(9 v^2+3 v-2\right) x^3+\frac{1}{2} \left(3 v^3+6 v^2-4
   v+1\right) x^2 + \\ &\left(2 v^3-v^2\right) x+\frac{1}{2} (v+3) x^6+(3
   v+2) x^5+\frac{x^7}{2},\end{split}\end{equation*}
the canonical form of the family $Y(0,1,1,v)$ is
\begin{equation*}\begin{split}   &V(v,X,Y,Z)= (v+8) X^3 Z-(4 v+2) X^2 Y^5 Z+v X^2 Y^2+(6 v+25) X^2 Z^2+(8-2 v) X Y
   Z^2 + \\&(12 v+36) X Z^3+(4 v+5) Y^2 Z^2+(9-4 v) Y Z^3+(9 v+20)
   Z^4+X^4+2 X^2 Y Z+Y^4-2 Y^3 Z=0.\end{split}\end{equation*}


\begin{thebibliography}{12}



\bibitem{Magma}
W. Bosma, J. Cannon, and C. Playoust, {\em The Magma algebra
system. I. The user language}, J. Symbolic Comput., 24 (1997),
235--265.

\bibitem{CEFS}
A. Chiodo, D. Eisenbud, G. Farks, and F. O. Schreyer, {\em	Syzygies of torsion bundles and the geometry of the level $\ell$ modular variety over $M_g$}, Inventiones Mathematicae 194 (2013), 73-118.


\bibitem{HenriCohen} H. Cohen, \textit{Advanced Topics in Computational Number Theory},
GTM193, Springer.


\bibitem{nE1}
N. D. Elkies, {\em Shimura curve computations}. Algorithmic number theory (Portland, OR, 1998), 1�47, Lecture Notes in Comput. Sci., 1423, Springer, Berlin, 1998, 1-47.

\bibitem{nE2}
\bysame, {\em Shimura curves for level-$3$ subgroups of the $(2,3,7)$ triangle group, and some other examples.} Algorithmic number theory, 302�316, Lecture Notes in Comput. Sci., 4076, Springer, Berlin, 2006.

\bibitem{nE3}
\bysame, {\em Shimura curve computations via $K3$ surfaces of N\'eron-Severi rank at least $19$.} Algorithmic number theory, 196--211, Lecture Notes in Comput. Sci., 5011, Springer, Berlin, 2008.



\bibitem{EK}
\bysame  and A. Kumar, {\em $K3$ surfaces and equations for Hilbert modular surfaces}, arXiv:1209.3527v1, (2012).

\bibitem{JE}
J. Ellenberg, {\em Endomorphism algebras of Jacobians}, Advances
in Mathematics, Vol. 162, No. 2 (2001), 243--271.



\bibitem{vdG}
G. van der Geer, {\em Hilbert modular surfaces}, Ergebnisse der Mathematik und ihrer Grenzgebiete 16. Springer-Verlag, Berlin, 1988.


\bibitem{eG}
E. Z. Goren, Lectures on Hilbert modular varieties and modular forms, With the assistance of Marc-Hubert Nicole, CIRM Monograph Series, 14, Amer. Math. Soc. 2002.

\bibitem{dG}
D. Gruenewald, {\em Explicit algorithms for Humbert surfaces.} thesis, U. Sydney (2009), available at http://echidna.maths.usyd.edu.au/$\sim$davidg/.

\bibitem{Gr}
P. A. Griffiths, {\em On the periods of certain rational integrals. I, II.} Ann. of Math. (2) 90 (1969), 460--495; ibid. (2) 90 1969 496�-541.


\bibitem{HM}
K. Hashimoto and N. Murabayashi, {\em  Shimura curves as intersections of Humbert surfaces and defining equations of QM-curves of genus two. }Tohoku Math. J. (2) 47 (1995), no. 2, 271--296.

\bibitem{HH}
J. W. Hoffman and  H. He, {\em Picard groups of Siegel modular $3$-folds
and $\theta$-liftings.} Journal of Lie Theory, (2012)(3), 769--801.



\bibitem{HW}
J. W. Hoffman and H. Wang, {\em $7$-gons and genus $3$ hyperelliptic
curves},  Revista de la Real Academia de Ciencias Exactas, Fisicas
y Naturales. Serie A. Matem\`aticas, 107 (2013), 35-52.

\bibitem{HLSW}
J. W. Hoffman, Z. Liang, Y. Sakai and H. Wang,
{\em Genus $3$  curves whose Jacobians have endomorphisms by  $\QQ (\zeta _7 +\bar{\zeta}_7 )$}, to appear.



\bibitem{jfM}
J. F. Mestre, {\em Courbes hyperelliptiques {\`a} multiplications
r{\'e}elles}, C. R. Acad. Sci. Paris, S«er. I Math., 307 (1988),
721--724.

\bibitem{jfM2}
\bysame, {\em Courbes hyperelliptiques {\`a} multiplications
r{\'e}elles}, Progr. Math., 89 (1991), 193--208.

\bibitem{dM}
D. Mumford, {\em Abelian Varieties}, Tata Institute of Fundamental
Research Studies in Mathematics, No. 5, Published for the Tata
Institute of Fundamental Research, Bombay; Oxford University
Press, London 1970.


\bibitem{pari}
The PARI~Group, {\em PARI/GP, version {\tt 2.5.0}}, 2011,
Bordeaux, available from http://pari.math.u-bordeaux.fr/

\bibitem{bR}
B. Runge, {\em  Endomorphism rings of abelian surfaces and
projective models of their moduli spaces.} Tohoku Math. J. (2) 51 (1999), no. 3, 283--303.


\bibitem{Sage}
SAGE Mathematics Software, Version 4.6, http://www.sagemath.org/

\bibitem{yS}
Y. Sakai, {\em Construction of genus two curves with real
multiplication by Poncelet's theorem.} (2010) dissertation, Waseda
University.

\bibitem{Shi}
I. Shimada, {\em A construction of algebraic curves whose Jacobians have non-trivial endomorphisms.} Comment. Math. Univ. St. Paul, 43(1) (1994), 25--34.


\bibitem{gS}
G. Shimura, {\em On analytic families of polarized abelian
varieties and automorphic functions.} Ann. of Math. (2) 78 1963,
149--192.

\bibitem{gS2}
\bysame, Abelian varieties with complex multiplication and modular
functions. Princeton Mathematical Series, 46. Princeton University
Press, Princeton, NJ, 1998.

\bibitem{jV1}
J. Voight, {\em Quadratic forms and quaternion algebras: algorithms and arithmetic}, thesis, Berkeley, 2005,
available at http://www.math.dartmouth.edu/$\sim$jvoight/articles/thesis.pdf.

\bibitem{jV2}
\bysame, {\em Shimura curves of genus at most two.} Math. Comp. 78 (2009), no. 266, 1155--1172.


 \bibitem{aW}
 A. Weil, {\em  Vari{\'e}t{\'e}s ab{\'e}liennes et courbes alg{\'e}briques},
 Hermann \& Cie., Paris, 1948.

\bibitem{Math}
Wolfram Research, Inc., {\sc Mathematica}, Version 7.0, Champaign,
IL (2008).





\end{thebibliography}
\end{document}